\documentclass[onefignum,onetabnum]{siamart190516}
\usepackage{amsmath}
\usepackage{amssymb}
\usepackage{enumerate}
\usepackage{mathrsfs}
\usepackage{cases}
\usepackage{color}
\usepackage{tikz}
\usepackage{makecell}
\usepackage{multirow}
\usepackage{siunitx}
\usepackage{algorithmic,algorithm}
\usepackage{graphicx}
\usepackage{mathtools}
\definecolor{darkblue}{rgb}{0.0, 0.0, 0.55}
\definecolor{bordeaux}{rgb}{0.34, 0.01, 0.1}
\usepackage{pgfplots}
\usepackage{scalefnt}
\usepgfplotslibrary{fillbetween}

\definecolor{color1}{RGB}{145,30,180}
\definecolor{color2}{RGB}{245,130,48}
\definecolor{color3}{RGB}{230,25,75}
\definecolor{lightgreen}{RGB}{144, 238, 144}

\newsiamthm{example}{Example}

\def\R{{\mathbb{R}}}

\def\N{{\mathbb{N}}}

\def\x{{\mathbf{x}}}

\def\S{{\mathbb{S}}}

\def\Tr{\hbox{\rm{Tr}}}

\def\grad{\hbox{\rm{grad}}}
\def\hess{\hbox{\rm{Hess}}}

\def\rank{\hbox{\rm{rank}}}

\def\cA{{\mathcal{A}}}
\def\cB{{\mathcal{B}}}
\def\cM{{\mathcal{M}}}
\def\cN{{\mathcal{N}}}
\def\Diag{\hbox{\rm{Diag}}}

\DeclareMathOperator*{\argmax}{arg\,max}

\usepackage{lipsum}
\usepackage{amsfonts}
\usepackage{graphicx}
\usepackage{epstopdf}
\usepackage{algorithmic}
\ifpdf
  \DeclareGraphicsExtensions{.eps,.pdf,.png,.jpg}
\else
  \DeclareGraphicsExtensions{.eps}
\fi

\newsiamremark{remark}{Remark}
\newsiamremark{hypothesis}{Hypothesis}
\crefname{hypothesis}{Hypothesis}{Hypotheses}
\newsiamthm{claim}{Claim}
\newsiamthm{assumption}{Assumption}

\headers{Solving low-rank semidefinite programs via manifold optimization}{Jie Wang, Liangbing Hu}

\title{Solving low-rank semidefinite programs via manifold optimization
}

\author{Jie Wang\thanks{Academy of Mathematics and Systems Science, Chinese Academy of Sciences, Beijing, China
  (\email{wangjie212@amss.ac.cn})} \and Liangbing Hu\thanks{Nanjing Research Institute of Electronics Technology, Nanjing, China
  (\email{huliangbing2000@163.com})}}

\usepackage{amsopn}
\DeclareMathOperator{\diag}{diag}

\ifpdf
\hypersetup{
  pdftitle={An Example Article},
  pdfauthor={D. Doe, P. T. Frank, and J. E. Smith}
}
\fi

\begin{document}

\maketitle

\begin{abstract}
We propose a manifold optimization approach to solve linear semidefinite programs (SDP) with low-rank solutions, with an emphasis on SDP relaxations for polynomial optimization problems. This approach incorporates the inexact augmented Lagrangian method (ALM) and the Burer-Monteiro factorization, and features the self-adaptive strategies for updating the factorization size and the penalty parameter.
We establish global convergence of the inexact ALM, despite the non-convexity brought by the Burer-Monteiro factorization. We further provide a practical algorithm building on the inexact ALM, and along with the algorithm we release an open-source SDP solver {\tt ManiSDP}.
Comprehensive numerical experiments demonstrate that {\tt ManiSDP} achieves state-of-the-art in terms of efficiency, accuracy, and scalability, and is faster than several advanced SDP solvers ({\tt MOSEK}, {\tt SDPLR}, {\tt SDPNAL+}, {\tt STRIDE}) by up to orders of magnitudes on a variety of linear SDPs. The largest SDP solved by {\tt ManiSDP} (in about 8.5 hours with maximal KKT residue 3.5e-13) is the second-order moment relaxation of a binary quadratic program with $120$ variables, which has matrix dimension $7261$ and contains $17,869,161$ affine constraints.
\end{abstract}



\section{Introduction}
In this paper, we aim to efficiently solve the following semidefinite programming (SDP) problem:
\begin{equation}\label{sdp}
\begin{cases}
\inf\limits_{X\succeq0}&\langle C, X\rangle\\
\,\,\rm{s.t.}&\mathcal{A}(X)=b, \mathcal{B}(X)=d,
\end{cases}\tag{SDP}
\end{equation}
where $\mathcal{A}:\S_n\rightarrow\R^m$, $\mathcal{B}:\S_n\rightarrow\R^l$ are linear maps ($\S_n$ denotes the set of $n\times n$ symmetric matrices), $b\in\R^m,d\in\R^l$.
In \eqref{sdp} the linear constraints $\cA(X)=b$ are arbitrary while the linear constraints $\cB(X)=d$, if present, are assumed to define certain manifold structure. Moreover, we assume that \eqref{sdp} admits a low-rank solution $X^{\star}$, i.e., $\rank(X^{\star})\ll n$.

Such \eqref{sdp} arises from a diverse set of fields, e.g., systems and control \cite{parrilo2003semidefinite}, signal processing \cite{de2008design,luo2010semidefinite}, optimal power flow \cite{bai2008semidefinite}, matrix completion \cite{candes2012exact}, quantum steering \cite{cavalcanti2016quantum}, computer vision \cite{yang2022certifiably}, to just name a few; see also the surveys \cite{vandenberghe1996semidefinite,wolkowicz2012handbook} and references therein.
It also serves as tractable convex relaxations for many difficult (usually NP-hard) non-convex optimization problems, e.g., quadratically constrained quadratic programs (QCQP) \cite{luo2010semidefinite}, combinatorial optimization problems \cite{goemans1997semidefinite}, polynomial optimization problems \cite{Las01}.

In view of the wide applications, a great deal of effort has been dedicated to solving \eqref{sdp} over the past decades and many practical algorithms have been developed from various angles.
For small/medium-scale SDPs, interior-point methods are believed to be the most accurate, efficient and robust algorithms \cite{andersen2003implementing,toh1999sdpt3}. However, for large-scale SDPs interior-point methods are no longer reliable because of their extensive memory occupation and high computational cost in solving a large and dense linear system at each iteration.
Aiming to tackle large-scale SDPs, Yang, Sun, and Toh proposed an augmented Lagrangian algorithm combined with the semismooth Newton method \cite{yang2015sdpnal}. The related solver {\tt SDPNAL+} has been shown to have good scalability on {\em degenerate} SDPs \cite{sun2020sdpnal}.
There are also quite a few attempts to overcome the memory issue by relying on first-order methods.
For instance, Wen et al. proposed an alternating-direction augmented Lagrangian method (ALM) \cite{wen2010alternating}. The main drawback of first-order methods is that they can hardly achieve high accuracy.

In many cases, large-scale SDPs possess certain structures, e.g., chordal sparsity, constant trace, admitting low-rank solutions. Such structures can be exploited to speed up computation either in interior-point methods or in first-order methods. Zhang and Lavaei proposed to exploit chordal sparsity for interior-point methods via dualized clique tree conversion \cite{zhang2021sparse}. Zheng et al. \cite{zheng2020chordal} and Garstka et al. \cite{garstka2021cosmo} proposed to exploit chordal sparsity in a framework of the alternating direction methods of multipliers (ADMM).
By exploiting the constant trace property, Helmberg and Rendl proposed a spectral bundle method for solving \eqref{sdp} \cite{helmberg2000spectral}. More recently, Yurtsever et al. proposed to exploit the constant trace property in a conditional-gradient-based augmented Lagrangian framework combined with a matrix sketching technique \cite{yurtsever2019scalable}.
The property of admitting low-rank solutions can be exploited in the framework of interior-point methods \cite{bellavia2021relaxed,habibi2021barrier,zhang2017modified} or in the framework of operator splitting methods \cite{souto2022exploiting}.
For SDPs arising as convex relaxations of polynomial optimization problems and admitting rank-one solutions, Yang et al. proposed a projected gradient method accelerated by local search \cite{yang2022inexact}, which is scalable on a variety of SDPs emerging from computer vision problems \cite{yang2022certifiably}.

Another notable approach taking the low-rank property into account is to perform a factorization $X=YY^{\intercal}$ with $Y\in\R^{n\times p},p\ll n$ so that \eqref{sdp} becomes a non-convex QCQP (called the factorized problem), which is called the {\em Burer-Monteiro factorization} in the literature \cite{burer2003nonlinear,burer2005local}, leading to a low-rank SDP solver {\tt SDPLR}. For the factorized problem, nonlinear programming tools can find a second-order critical point at much less cost, and for almost all cost matrices $C$, under mild conditions this second-order critical point is guaranteed to be globally optimal whenever $\frac{p(p+1)}{2}>m$ \cite{boumal2016non,boumal2020deterministic,cifuentes2021burer}. Hence for the approach being efficient, $m$ is required not to be too large. Recent treatments of this approach for certifiably correct machine perception and the graph equipartition problem could be found in \cite{rosen2021scalable, tang2024solving}. See also \cite{han2024low,huang2024suboptimality,monteiro2024low,tang2023feasible} for very recent developments of this approach.

\vspace{0.5em}
{\bf Related literature.} Following Burer and Monteiro's idea, Journ\'ee et al. reformulated \eqref{sdp} as a nonlinear program on a manifold under the assumption that there is no arbitrary linear constraint (i.e., $m=0$) and the feasible set $\cN_p=\{Y\in\R^{n\times p}\mid\mathcal{B}(YY^{\intercal})=d\}$ is a smooth manifold \cite{journee2010low}. Consequently, the nonlinear program becomes a Riemannian optimization problem which can be efficiently solved by off-the-shelf Riemannian optimization tools. They also provided a strategy for escaping from saddle points so that global convergence can be guaranteed. See \cite{2204.14067} for an extension of this approach to the regularized convex matrix optimization problem and \cite{rosen2019se} to the special Euclidean synchronization problem. A drawback of Journ\'ee et al.'s approach is that they did not deal with linear constraints that do not define a manifold.
The first work to treat constrained optimization on Riemannian manifolds using an augmented Lagrangian framework appears to be \cite{liu2020simple}. Later, Zhou et al. developed an ALM for solving a class of manifold optimization problems with nonsmooth objective functions and nonlinear constraints \cite{zhou2022semismooth}. Building on constrained Riemannian optimization, Wang et al. recently extended Journ\'ee et al.'s approach for solving SDPs with nonsmooth objective functions and arbitrary linear constraints $\cA(X)=b$ within an augmented Lagrangian framework in which a Riemannian semismooth Newton method is employed to solve the ALM subproblem \cite{wang2023decomposition}.

\vspace{0.5em}
{\bf Contributions.} Our contributions are as follows.

$\bullet$ We present a manifold optimization approach for solving \eqref{sdp} by adopting the Burer-Monteiro factorization and the idea of Journ\'ee et al. More concretely, our approach employs an ALM framework to handle the linear constraints $\cA(X)=b$ and applies the Burer-Monteiro factorization $X=YY^{\intercal}$ to the ALM subproblem in order to exploit the low-rank property. The ALM subproblem is then recast as a manifold optimization problem on $\cN_p$ which could be solved with efficient manifold optimization methods. To circumvent the non-convexity introduced by the Burer-Monteiro factorization, we design an effective strategy (inspired by Journ\'ee et al.) to escape from saddle points. Then under certain conditions, we establish the global convergence of the proposed inexact ALM.

$\bullet$ To further enhance the practical performance of the approach, we propose self-adaptive strategies for updating the factorization size and the penalty parameter.
Specifically, (i) we dynamically adjust the factorization size $p$ so that the decision variables of the ALM subproblem is as few as possible; (ii) we propose a \emph{simple} strategy to self-adaptively increase or decrease the penalty parameter so that it will not become too large (note that the ALM subproblem with a large penalty parameter is more difficult to solve).

$\bullet$ We provide a practical algorithm building on the inexact ALM, where the ALM subproblem is solved only to gain some descent (without optimality requirements) and the procedure of escaping from saddle points is performed only once within each outer iteration of the ALM. Surprisingly and intriguingly, global convergence still occurs in all numerical examples examined in this paper.

$\bullet$ As another main contribution, we release an open-source low-rank SDP solver named {\tt ManiSDP} that implements the proposed algorithm. Extensive numerical experiments were performed to benchmark the solver, which demonstrate that {\tt ManiSDP} is accurate, efficient, and scalable on a variety of SDPs with low-rank solutions, and (substantially) outperforms a few popular SDP solvers.

Although both incorporate the ALM framework, the Burer-Monteiro factorization and manifold optimization, this work differs from \cite{wang2023decomposition} in fourfold: (1) This work targets at solving large-scale linear SDPs, especially those arising from \emph{the moment-SOS hierarchy} of polynomial optimization problems while \cite{wang2023decomposition} primarily aims to solve nonlinear and nonsmooth SDPs. Accordingly, this work adopts the Riemannian trust-region method to solve the ALM subproblem while \cite{wang2023decomposition} adopts a Riemannian semismooth Newton method. (2) This work introduces a \emph{single-step} strategy for escaping from saddle points within each ALM outer iteration, yielding a practical algorithm with superior performance. (3) This work distinguishes itself from conventional ALM implementations and \cite{wang2023decomposition} through a
self-adaptive strategy that enables reduction of the penalty parameter and improves the performance of the algorithm a lot.
(4) Along with an explicit algorithm, this work releases an open-source low-rank SDP solver.

The rest of the paper is organized as follows. In Section \ref{preliminaries}, we collect notations and some preliminary results. In Section \ref{alfw}, we give the augmented Lagrangian framework with the Burer-Monteiro factorization. In Section \ref{subproblem}, we describe some computational details on solving the ALM subproblem with the Riemannian trust-region method, discuss how to escape from saddle points, and prove global convergence of the inexact ALM. In Section \ref{algorithms}, we describe the strategies of updating the factorization size and the penalty parameter, and present the practical algorithm. Results of numerical experiments are provided in Section \ref{experiments}. Conclusions are made in Section \ref{conclusions}.

\section{Notation and preliminaries}\label{preliminaries}
$\R$ (resp. $\R^+$, $\N$) denotes the set of real numbers (resp. positive real numbers, nonnegative integers). For a positive integer $n$, let $[n]\coloneqq\{1,2,\ldots,n\}$. Let $\S_n$ (resp. $\S^+_n$) denote the set of (resp. positive semidefinite/PSD) symmetric matrices of size $n$. We use $\Tr(A)$ (resp. $A^{\intercal}$) to denote the trace (resp. transpose) of a matrix $A\in\R^{n\times m}$. For two matrices $A,B\in\R^{n\times m}$, the inner product is defined as $\langle A, B\rangle=\Tr(A^{\intercal}B)$. For $A\in\R^{n\times n}$, $\diag(A)$ denotes the diagonal of $A$, $\Diag(A)$ denotes the diagonal matrix with the same diagonal as $A$, and $\lambda_{\min}(A),\lambda_{\max}(A)$ denote the smallest, largest eigenvalues of $A$, respectively. For a vector $v$, $\|v\|$ is the $2$-norm of $v$ and for a matrix $A$, $\|A\|$ is the Frobenius norm of $A$.
For a function $f(x)$, we write $\nabla f(x)$ (resp. $\grad\,f(x)$) for the Euclidean (resp. Riemannian) gradient, and write $\nabla^2 f(x)[u]$ (resp. $\hess\,f(x)[u]$) for the Euclidean (resp. Riemannian) Hessian acting on $u$. For a set $E$, $|E|$ denotes its cardinality.

Let us consider \eqref{sdp} with $\mathcal{A}(X)\coloneqq(\langle A_i, X\rangle)_{i=1}^m$ and $\mathcal{B}(X)\coloneqq(\langle B_i, X\rangle)_{i=1}^l$ for $A_i,B_i\in\S_n$. Let $\cA^*:\R^m\rightarrow\S_n$ be the adjoint operator of $\cA$ defined as $\cA^*(y)\coloneqq\sum_{i=1}^my_iA_i$ for $y\in\R^m$; similarly, let $\cB^*:\R^l\rightarrow\S_n$ be the adjoint operator of $\cB$ defined as $\cB^*(z)\coloneqq\sum_{i=1}^lz_iB_i$ for $z\in\R^l$. Throughout the paper, we assume that strong duality holds for \eqref{sdp}.
\begin{lemma}\label{lm:optimality}
A matrix $X\in\S_n$ is a minimizer of \eqref{sdp} if and only if there exist Lagrange multipliers $y\in\R^m$ and $z\in\R^l$ such that
\begin{subequations}\label{optimality}
\begin{align}
\mathcal{A}(X)&=b,\\
\mathcal{B}(X)&=d,\\
X&\succeq0,\\
S\coloneqq C-\mathcal{A}^*(y)-\mathcal{B}^*(z)&\succeq0,\\
XS&=0.
\end{align}
\end{subequations}
\end{lemma}
\begin{proof}
These are the standard KKT conditions for \eqref{sdp}.
\end{proof}

\section{An augmented Lagrangian framework with the Burer-Monteiro factorization}\label{alfw}
In this section, we introduce an augmented Lagrangian framework combined with the Burer-Monteiro factorization for solving \eqref{sdp}.
Recall that the constraints $\mathcal{B}(X)=d$ in \eqref{sdp} define a certain manifold structure (which will be rigorously defined later). Let us denote
\begin{equation}\label{maniX}
\cM\coloneqq\{X\in\R^{n\times n}\mid\mathcal{B}(X)=d,X\succeq0\}.
\end{equation}
Note that if the constraints $\mathcal{B}(X)=d$ are not present (i.e., $l=0$), then $\cM=\S^+_n$. Other typical choices of $\cM$ are
\begin{align}
\cM&=\{X\in\R^{n\times n}\mid X\succeq0,\Tr(X)=1\},\tag{Unit-trace}\label{maniX1}\\
\cM&=\{X\in\R^{n\times n}\mid X\succeq0,\diag(X)=\mathbf{1}\}.\tag{Unit-diagonal}\label{maniX2}
\end{align}
\begin{remark}
For the case that $X$ has a constant trace $c$, we can scale $X$ by the factor $\frac{1}{c}$ to match the \eqref{maniX1} case.
\end{remark}
Then \eqref{sdp} can be equivalently written as
\begin{equation}\label{sdp1}
\begin{cases}
\inf\limits_{X\in\cM}&\langle C, X\rangle\\
\,\,\,\rm{s.t.}&\mathcal{A}(X)=b.
\end{cases}\tag{SDP-M}
\end{equation}

For $p\in[n]$, we denote by $\cN_p$ the set of matrices after applying the Burer-Monteiro factorization to $\cM$, i.e., 
\begin{equation}\label{maniY}
\cN_p\coloneqq\{Y\in\R^{n\times p}\mid YY^{\intercal}\in\cM\}=\{Y\in\R^{n\times p}\mid\mathcal{B}(YY^{\intercal})=d\}.
\end{equation}
We call $p$ the \emph{factorization size}.
For $\cM=\S^+_n$ or $\cM$ in \eqref{maniX1}, \eqref{maniX2}, the corresponding $\cN_p$ are
\begin{subequations}
\begin{align}
\cN_p&=\{Y\in\R^{n\times p}\},\tag{Euclidean}\label{maniYs1}\\
\cN_p&=\{Y\in\R^{n\times p}\mid \|Y\|=1\},\tag{Sphere}\label{maniYs2}\\
\cN_p&=\{Y\in\R^{n\times p}\mid \|Y(i,:)\|=1, i=1,\ldots,n\},\tag{Oblique}\label{maniYs3}
\end{align}
\end{subequations}
where $Y(i,:)$ stands for the $i$-th row of $Y$.

We further make the following assumptions on \eqref{sdp1}:

\begin{assumption}\label{assump1}
\eqref{sdp1} admits a low-rank optimal solution.
\end{assumption}

\begin{assumption}\label{assump2}
$\cN_p$ is a submanifold embedded in the Euclidean space $\R^{n\times p}$.
\end{assumption}

\begin{assumption}\label{assump3}
Either $l=0,1$, or the matrices $\{B_i\}_{i=1}^l$ satisfy $B_iB_j=0$ for any $1\le i\ne j\le l$.
\end{assumption}

\begin{assumption}\label{assump4}
$d_i\ne0$ for $i=1,\ldots,l$.
\end{assumption}

It is clear that for $\cN_p$ in \eqref{maniYs1} or \eqref{maniYs2} with $p\ge1$, or $\cN_p$ in \eqref{maniYs3} with $p\ge2$, Assumptions \ref{assump2}--\ref{assump4} are satisfied.

\begin{remark}
Assumptions \ref{assump3}--\ref{assump4} are not essential and can be weakened to the assumption that the matrices $\{B_iY\}_{i=1}^l$ are linearly independent in $\R^{n\times p}$ for all $Y\in\cN_p$; we refer the reader to \cite{boumal2020deterministic} for detailed discussions. In this paper, we utilize assumptions \ref{assump3}--\ref{assump4} to obtain a simple closed form of the dual variables $z$ associated to the manifold constraints.
\end{remark}



The augmented Lagrangian function associated with \eqref{sdp1} is defined by
\begin{equation}\label{alf}
L_{\sigma}(X,y)=\langle C, X\rangle-y^{\intercal}(\mathcal{A}(X)-b)+\frac{\sigma}{2}\|\mathcal{A}(X)-b\|^2.
\end{equation}
As for usual constrained optimization problems, \eqref{sdp1} can be solved within an inexact ALM which is presented in Algorithm \ref{alg0} (cf. \cite{wang2023decomposition}).

\begin{algorithm}
\renewcommand{\algorithmicrequire}{\textbf{Input:}}
\renewcommand{\algorithmicensure}{\textbf{Output:}}
\caption{Inexact ALM template}\label{alg0}
	\begin{algorithmic}[1]
   	\REQUIRE $\cA,b,\cB,d,C,\sigma_0>0,\sigma_{\min}>0,\sigma_{\max}>0$
   	\ENSURE $(X,y,S)$
		\STATE $k \gets 0$, $y^0 \gets 0$;
		\WHILE{stopping criteria do not fulfill}
		\STATE Solve the ALM subproblem \eqref{subpX} inexactly to obtain an approximate minimizer $X^{k+1}$;
		\STATE $y^{k+1} \gets y^k-\sigma_k(\mathcal{A}(X^{k+1})-b)$;
            \STATE $S^{k+1}\gets\nabla\Phi_k(X^{k+1})-\mathcal{B}^*(z)$, where $z$ is given by \eqref{multiplierX} with $X\coloneqq X^{k+1}$;
            \STATE Determine $\sigma_{k+1}\in[\sigma_{\min},\sigma_{\max}]$ according to some policy;
		\STATE $k \gets k+1$;
		\ENDWHILE
		\RETURN $(X^{k},y^{k},S^{k})$;
	\end{algorithmic}
\end{algorithm}

At the $k$-th iteration of Algorithm \ref{alg0}, we need to solve the ALM subproblem
\begin{equation}\label{subpX}
\min_{X\in\cM}\Phi_k(X)\coloneqq L_{\sigma_{k}}(X,y^{k}).
\end{equation}
The following lemma characterizes the optimality conditions of \eqref{subpX}.
\begin{lemma}\label{lm:suboptimality}
A matrix $X\in\cM$ is a minimizer of \eqref{subpX} if and only if there exist Lagrange multipliers $z\in\R^l$ such that
\begin{subequations}
\begin{align}
S\coloneqq \nabla\Phi_k(X)-\mathcal{B}^*(z)&\succeq0,\label{optimality3}\\ 
XS&=0, \label{optimality4}
\end{align}
\end{subequations}
where $\nabla\Phi_k(X)=C+\sigma_k\mathcal{A}^*(\mathcal{A}(X)-b-y^k/\sigma_k)$.
Furthermore, under Assumptions \ref{assump3}--\ref{assump4}, the vector $z$ is unique and has a closed-form expression:
\begin{equation}\label{multiplierX}
z_i\coloneqq\frac{\Tr(B_i\nabla\Phi_k(X)X)}{\Tr(B_i^2X)}\text{ for }i=1,\ldots,l.
\end{equation}
\end{lemma}
\begin{proof}
As \eqref{subpX} is convex, $X$ is a minimizer if and only if the KKT conditions hold, i.e., there exists $z\in\R^l$ such that
\begin{subequations}
\begin{align}
\mathcal{B}(X)&=d, \label{optimality1}\\
X&\succeq0, \label{optimality2}\\
S=\nabla\Phi_k(X)-\mathcal{B}^*(z)&\succeq0,\\
XS=X(\nabla\Phi_k(X)-\mathcal{B}^*(z))&=0.
\end{align}
\end{subequations}
Since $X\in\cM$, \eqref{optimality1} and \eqref{optimality2} are valid by definition, and so $X$ is a minimizer if and only if \eqref{optimality3} and \eqref{optimality4} hold for some $z\in\R^l$.
To get the closed-form expression of $z$, we note that $XS=0$
can be rewritten as
\begin{equation}\label{sec3:eq1}
\nabla\Phi_k(X)X-\sum_{j=1}^lB_jXz_j=0.
\end{equation}
For each $i\in[l]$, multiplying \eqref{sec3:eq1} by $B_i$ and noting $B_iB_j=0$ for $i\ne j$, we get
\begin{equation}\label{sec3:eq3}
B_i\nabla\Phi_k(X)X-B_i^2Xz_i=0.
\end{equation}
Let $X=YY^{\intercal}$ for some $Y\in\cN_p$. We have
\begin{equation*}
\Tr(B_i^2X)=\Tr(Y^{\intercal}B_i^2Y)=\|B_iY\|^2\ne0,
\end{equation*}
where the last inequality is because $B_iY$ cannot be a zero matrix since $d_i\ne0$.
Then \eqref{multiplierX} follows by taking the trace of \eqref{sec3:eq3}.
\end{proof}

\begin{remark}
For $\cM$ in \eqref{maniX1}, \eqref{multiplierX} becomes
\begin{equation}\label{sec4:eq3}
z=\Tr(\nabla\Phi_k(X)X),
\end{equation}
and for $\cM$ in \eqref{maniX2}, \eqref{multiplierX} becomes
\begin{equation}\label{sec4:eq2}
z=\diag(\nabla\Phi_k(X)X).
\end{equation}
\end{remark}


To exploit the fact that \eqref{sdp1} admits a low-rank optimal solution (Assumption \ref{assump1}), we now apply the Burer-Monteiro factorization to $\Phi_k(X)$ and define $\Psi_k(Y)=\Phi_k(YY^{\intercal})$ for $Y\in\R^{n\times p}$. Consequently, the convex subproblem \eqref{subpX} becomes the non-convex factorized subproblem:
\begin{equation}\label{subpY}
\min_{Y\in\cN_p}\Psi_k(Y)=\langle C, YY^{\intercal}\rangle-(y^k)^{\intercal}(\mathcal{A}(YY^{\intercal})-b)+\frac{\sigma_k}{2}\|\mathcal{A}(YY^{\intercal})-b\|^2.\tag{ALMS-$k$}
\end{equation}
Then, we solve the non-convex factorized subproblem \eqref{subpY} on the manifold $\cN_p$ (Assumption \ref{assump2}) instead of solving the convex subproblem \eqref{subpX} on $\cM$. 

\section{A manifold optimization approach}\label{subproblem}
Since $\cN_p$ is assumed to be a smooth manifold, we can solve the non-convex factorized subproblem \eqref{subpY} by off-the-shelf efficient manifold optimization methods. Here, we choose the Riemannian trust-region method. We elaborate on the rationale behind this choice as follows. Existing work \cite{cui2019r,liao2025inexact} has demonstrated that the inexact ALM for solving linear SDPs can achieve linear convergence, provided that the ALM subproblem is solved with sufficient accuracy. This finding indicates that to retain the fast linear convergence of the inexact ALM and to obtain optimal solutions of high accuracy, it is more preferable to employ a second-order method for solving the ALM subproblem. Among second-order methods for smooth optimization\footnote{Since we focus on linear SDPs in this paper, the augmented Lagrangian function is a smooth function and it is then more natural to employ a smooth optimization algorithm for solving the ALM subproblem. That is why we choose the Riemannian trust-region method rather than the Riemannian semismooth Newton method method (which is designed for nonsmooth optimization) proposed in \cite{wang2023decomposition}.}, we adopt the Riemannian trust-region method as it not only guarantees global convergence but also exhibits a superlinear (or even quadratic) local convergence rate \cite{AbsBakGal2007-FoCM}, thereby combining robustness with exceptional computational efficiency in ALM implementations.

\subsection{The Riemannian trust-region method}
In this subsection, we calculate the ingredients that are necessary to perform optimization on the manifold $\cN_p$ via the Riemannian trust-region method. For a detailed introduction to the Riemannian trust-region method, we refer the reader to \cite{AbsBakGal2007-FoCM}. First of all, we note that at a point $Y\in\cN_p$, the tangent space of $\cN_p$ is given by
\begin{equation}
    T_Y\cN_p=\{U\in\R^{n\times p}\mid\cB(UY^{\intercal})=0\},
\end{equation}
and the normal space to $\cN_p$ is given by
\begin{equation}
    N_Y\cN_p=\left\{\cB^*(u)Y=\sum_{i=1}^lu_iB_iY\middle| u\in\R^l\right\}.
\end{equation}

\begin{lemma}\label{sec4:lm1}
Let Assumption \ref{assump3} hold. Let $Y$ be a point on $\cN_p$. The orthogonal projector $P_Y:\R^{n\times p}\rightarrow T_Y\cN_p$ is given by
\begin{equation}
    P_Y(U)=U-\cB^*(u)Y=U-\sum_{i=1}^lu_iB_iY, \text{ for }U\in\R^{n\times p},
\end{equation}
where $u\in\R^l$ is determined by
\begin{equation}\label{sec4:eq5}
    u_i=\frac{\Tr(B_iUY^{\intercal})}{\Tr(B_i^2YY^{\intercal})},\text{ for }i=1,\ldots,l. 
\end{equation}
\end{lemma}
\begin{proof}
Any matrix $U\in\R^{n\times p}$ admits a unique decomposition $U=U_{T_Y\cN_p}+U_{N_Y\cN_p}$,
where $U_{\mathcal{X}}$ belongs to the space $\mathcal{X}$. So the projection $P_Y(U)$ can be assumed to be of the form $P_Y(U)=U-\cB^*(u)Y$ for some $u\in\R^l$. Since $P_Y(U)\in T_Y\cN_p$, it holds
\begin{equation*}
    \cB(P_Y(U)Y^{\intercal})=\cB(UY^{\intercal})-\cB(\cB^*(u)YY^{\intercal})=0,
\end{equation*}
i.e., $\Tr(B_iUY^{\intercal})=\Tr(B_i\cB^*(u)YY^{\intercal})$ for $i=1,\ldots,l$. This yields \eqref{sec4:eq5} as $B_iB_j=0$ for $i\ne j$ by Assumption \ref{assump3}.
\end{proof}

\begin{proposition}[cf. \cite{wang2023decomposition}, Proposition 2.3]\label{prop}
Consider the non-convex factorized subproblem \eqref{subpY}. Let $X=YY^{\intercal}$, $\widetilde{S}=\nabla\Phi_k(X)$ and $S=\widetilde{S}-\cB^*(z)$ with $z$ being given in \eqref{multiplierX}. Then the \emph{Riemannian gradient} at $Y$ is given by
\begin{equation}\label{sec4:eq6}
\grad\,\Psi_k(Y)=2SY.
\end{equation}
For $U\in T_{Y}\cN_p$, let $\widetilde{H}=\nabla^2\Psi_k(Y)[U]=2(\widetilde{S}U+\sigma_k\cA^*(\cA(YU^{\intercal}+UY^{\intercal}))Y)$. Then the \emph{Riemannian Hessian} is given by
\begin{equation}\label{sec4:eq7}
\hess\,\Psi_k(Y)[U]=P_Y(\widetilde{H})-2\cB^*(z)U+2\cB^*(u)Y,
\end{equation}
where
\begin{equation}
    u_i=\frac{\Tr(B_i^2UY^{\intercal})z_i}{\Tr(B_i^2X)},\text{ for }i=1,\ldots,l. 
\end{equation}
In particular, for $\cN_p$ in \eqref{maniYs1}, we have
\begin{equation}
\hess\,\Psi_k(Y)[U]=\widetilde{H};
\end{equation}
for $\cN_p$ in \eqref{maniYs2}, we have
\begin{equation}
\hess\,\Psi_k(Y)[U]=\widetilde{H}-\Tr(\widetilde{H}Y^{\intercal})Y-2\Tr(\widetilde{S}X)U;
\end{equation}
for $\cN_p$ in \eqref{maniYs3}, we have
\begin{equation}
\hess\,\Psi_k(Y)[U]=\widetilde{H}-\Diag(\widetilde{H}Y^{\intercal})Y-2\Diag(\widetilde{S}X)U.
\end{equation}
\end{proposition}
\begin{proof}
First note $\nabla\Psi_k(Y)=2\widetilde{S}Y$. By (3) of \cite{absil2013extrinsic}, $\grad\,\Psi_k(Y)=P_Y(\nabla\Psi_k(Y))$ and thus we can write
\begin{equation*}
\grad\,\Psi_k(Y)=2\widetilde{S}Y-2\cB^*(z)Y
\end{equation*}
for some $z\in\R^l$. As $\grad\,\Psi_k(Y)\in T_Y\cN_p$, we have
\begin{equation}\label{sec4:eq4}
\cB(\grad\,\Psi_k(Y)Y^{\intercal})=2\cB((\widetilde{S}-\cB^*(z))X)=0.
\end{equation}
Solving \eqref{sec4:eq4} for $z$ we get exactly \eqref{multiplierX} and \eqref{sec4:eq6} then follows.

By (10) of \cite{absil2013extrinsic}, it holds
\begin{equation}
    \hess\,\Psi_k(Y)[U]=P_Y(\nabla^2\Psi_k(Y)[U])+\mathfrak{A}_Y(U,P_Y^{\perp}(\nabla\Psi_k(Y))),
\end{equation}
where $\mathfrak{A}_Y$ is the Weingarten map at $Y$ and $P_Y^{\perp}=I-P_Y$ is the orthogonal projector at $Y$ to $N_Y\cN_p$. Let $D_Y(\cdot)[U]$ be the directional derivative at $Y$ along $U$.
Then we have
\begin{align*}
\mathfrak{A}_Y(U,P_Y^{\perp}(\nabla\Psi_k(Y)))&=\mathfrak{A}_Y(U,2\cB^*(z)Y)\\
&=-P_Y(D_Y(2\cB^*(z)Y)[U])\\
&=-2P_Y(\cB^*(z)U)-2P_Y(\cB^*(D_{Y}(z)[U])Y)\\
&=-2\cB^*(z)U+2\cB^*(u)Y,
\end{align*}
where we have used the fact that $P_Y(\cB^*(z)U)=\cB^*(z)U-\cB^*(u)Y$ and $P_Y(\cB^*(u')Y)=0$ for any $u'\in\R^l$. \eqref{sec4:eq7} then follows.

The remaining conclusions of the proposition can be easily verified.
\end{proof}

The global optimality condition of \eqref{subpY} can be characterized in terms of positive semidefiniteness of the matrix $S$ (cf. \cite[Theorem 4]{journee2010low}).

\begin{proposition}\label{sec4:thm1}
Let $Y\in\cN_p$, $X=YY^{\intercal}\in\cM$, and $z$ be given by \eqref{multiplierX}. Let $S=\nabla\Phi_k(X)-\mathcal{B}^*(z)$. Then a stationary point $Y$ of the non-convex factorized subproblem \eqref{subpY} is a global minimizer if and only if $S\succeq0$.
\end{proposition}
\begin{proof}
Note that $Y$ is a global minimizer of \eqref{subpY} if and only if $X$ is a minimizer of the convex subproblem \eqref{subpX}. The fact that $Y$ is a stationary point implies $\grad\,\Psi_k(Y)=2SY=0$ and, consequently, $XS=0$. The conclusion then follows from Lemma \ref{lm:suboptimality}.
\end{proof}

As a corollary of Proposition \ref{sec4:thm1}, we obtain the following theorem.
\begin{theorem}\label{sec4:thm3}
A stationary point $Y\in\cN_p$ of the non-convex factorized subproblem \eqref{subpY} provides a minimizer $X=YY^{\intercal}$ of \eqref{sdp1} if and only if $\mathcal{A}(X)=b$ and $S=\nabla\Phi_k(X)-\mathcal{B}^*(z)\succeq0$ with $z$ being given by \eqref{multiplierX}.
\end{theorem}
\begin{proof}
This is immediate from Proposition \ref{sec4:thm1}.
\end{proof}

\subsection{Escaping from saddle points}\label{sec4-2}
Since the subproblem \eqref{subpY} is non-convex, in order to solve \eqref{subpY} to certain optimality, it is then crucial to design a strategy for escaping from saddle points. We next show that we can always compute a second-order descent direction $U$ to escape from saddle points whenever $S=\nabla\Phi_k(X)-\mathcal{B}^*(z)\nsucceq0$.

\begin{lemma}\label{sec4:lm2}
For any $U\in\R^{n\times p}$ satisfying $YU^{\intercal}=0$, it holds 
\begin{equation}
    \langle U,\hess\,\Psi_k(Y)[U]\rangle=2\Tr(U^{\intercal}SU).
\end{equation}
\end{lemma}
\begin{proof}
The assumption $YU^{\intercal}=0$ implies $U\in T_{Y}\cN_p$. By \eqref{sec4:eq7} and Lemma \ref{sec4:lm1}, $\widetilde{H}=2\widetilde{S}U$ with $\widetilde{S}=\nabla\Phi_k(X)$ and
\begin{align*}
\hess\,\Psi_k(Y)[U]&=2P_Y(\widetilde{S}U)-2\cB^*(z)U+2\cB^*(u)Y\\
&=2\widetilde{S}U-2\cB^*(u')Y-2\cB^*(z)U+2\cB^*(u)Y\\
&=2SU+2\cB^*(u-u')Y
\end{align*}
for $z,u$ in Proposition \ref{prop} and some $u'\in\R^l$. Thus,
\begin{align*}
\langle U,\hess\,\Psi_k(Y)[U]\rangle&=2\Tr(U^{\intercal}SU)+2\Tr(U^{\intercal}\cB^*(u-u')Y)=2\Tr(U^{\intercal}SU).
\end{align*}
\end{proof}

\begin{theorem}\label{sec5:thm1}
Suppose $S=\nabla\Phi_k(X)-\mathcal{B}^*(z)\nsucceq0$ with $z$ being given by \eqref{multiplierX}. Let $\delta\in\N$ be a positive number and let $V\in\R^{n\times\delta}$ be a matrix whose columns consist of eigenvectors corresponding to negative eigenvalues of $S$. Then $U\coloneqq[0_{n\times p}, V]$ is a second-order descent direction of \eqref{subpY} with $p\coloneqq p+\delta$ at the point $Y\coloneqq[Y, 0_{n\times \delta}]$, namely, $U$ satisfies
\begin{equation}
\langle U,\grad\,\Psi_k(Y)\rangle=0,\quad\langle U,\hess\,\Psi_k(Y)[U]\rangle<0.
\end{equation}
\end{theorem}
\begin{proof}
By construction, we have $YU^{\intercal}=0$. Therefore, by \eqref{sec4:eq6},
\begin{equation*}
\langle U,\grad\,\Psi_k(Y)\rangle=2\Tr(U^{\intercal}SY)=2\Tr(SYU^{\intercal})=0.
\end{equation*}
By Lemma \ref{sec4:lm2},
\begin{align*}
\langle U,\hess\,\Psi_k(Y)[U]\rangle=2\Tr(U^{\intercal}SU)=2\Tr(V^{\intercal}SV)<0.
\end{align*}
\end{proof}

\begin{remark}
The fact that an eigenvector corresponding to a negative eigenvalue of $S$ yields a second-order descent direction of \eqref{subpY} was first observed in \cite{journee2010low} where only one eigenvector corresponding to the smallest eigenvalue was used. In \cite{wang2023decomposition}, the authors used eigenvectors corresponding to all negative eigenvalues to obtain a descent direction.
\end{remark}

The following theorem adapted from \cite{journee2010low} (see also \cite{wang2023decomposition} for an extension to the nonsmooth case) tells us that by escaping from saddle points, we are capable of finding a global minimizer of the non-convex factorized subproblem \eqref{subpY}.
\begin{theorem}[\cite{journee2010low}, Theorem 7]\label{sec4:thm2}
A second-order critical point $Y\in\cN_p$ of the non-convex factorized subproblem \eqref{subpY} provides a minimizer $X=YY^{\intercal}$ of the convex subproblem \eqref{subpX} if it is rank deficient, i.e., $\rank\,Y<p$.
\end{theorem}
\begin{proof}
Let $\rank\,Y=r<p$. We can write $Y=Y'P^{\intercal}$ for some full-rank matrices $Y'\in\R^{n\times r}$ and $P\in\R^{p\times r}$. Let $P_{\perp}\in\R^{p\times(p-r)}$ be the matrix whose columns form an orthogonal basis of the orthogonal complement of the column space of $P$ such that $P^{\intercal}P_{\perp}=0$. Let $V\in\R^{n\times (p-r)}$ be an arbitrary matrix and let $U=VP_{\perp}^{\intercal}$. We have $UY^{\intercal}=VP_{\perp}^{\intercal}P(Y')^{\intercal}=0$. As $Y$ is a second-order critical point of the non-convex factorized subproblem \eqref{subpY}, $\hess\,\Psi_k(Y)$ is positive semidefinite. Then by Lemma \ref{sec4:lm2}, we have
\begin{equation*}
    0\le\langle U,\hess\,\Psi_k(Y)[U]\rangle=2\Tr(U^{\intercal}SU)=2\Tr(V^{\intercal}SV).
\end{equation*}
The arbitrarity of $V$ implies that $S$ is positive semidefinite and so by Proposition \ref{sec4:thm1}, $Y\in\cN_p$ is a global minimizer of \eqref{subpY}, which further implies that $X=YY^{\intercal}$ is a minimizer of \eqref{subpX}.
\end{proof}

\subsection{Global convergence}
We now establish the global convergence of the inexact ALM for solving \eqref{sdp1} assuming that the ALM subproblem \eqref{subpY} is solved to certain optimality. 

\begin{theorem}\label{gc-thm}
Let $\{\varepsilon_k\}_{k\in\N},\{\tau_k\}_{k\in\N}\subseteq\R^+$ satisfy $\sum_{i=0}^{\infty}\varepsilon_k<\infty$ and $\sum_{i=0}^{\infty}\tau_k<\infty$. Suppose that we use the following stopping criteria for the ALM subproblem \eqref{subpY}:
\begin{equation}\label{stopcri}
    \|\grad\,\Psi_k(Y)\|\le\varepsilon_k\text{ and }\lambda_{\min}(S)\ge-\tau_k,
\end{equation}
where $S$ is the dual variable to $X$ defined as in Proposition \ref{prop}. Assume that the sequence $\{Y^{k}\}_{k\in\N}$ is bounded. Let $(\hat{X},\hat{y},\hat{S})$ be a limit point of $\{(X^k,y^k,S^k)\}_{k\ge1}$. Then $(\hat{X},\hat{y},\hat{S})$ is a KKT point of \eqref{sdp1}, i.e., it satisfies \eqref{optimality} with $\hat{z}$ being given by
\begin{equation*}
\hat{z}_i\coloneqq\frac{\Tr(B_i(C-\cA^*(\hat{y}))\hat{X})}{\Tr(B_i^2\hat{X})}\text{ for }i=1,\ldots,l.
\end{equation*}
\end{theorem}

For the proof, we need the following two lemmas.

\begin{lemma}\label{sec5:lm1}
Suppose that $Y^{k+1}$ is an approximate minimizer of the ALM subproblem \eqref{subpY} so that the criteria \eqref{stopcri} is fulfilled.
Let $X^{k+1}\coloneqq Y^{k+1}(Y^{k+1})^{\intercal}$ and $(X^{\star},y^{\star},S^{\star})$ be a KKT point of \eqref{sdp1}. Then
\begin{equation}\label{sec5:eq5}
    \langle y^k-y^{\star}, \mathcal{A}(X^{k+1})-b\rangle\ge \sigma_k\|\mathcal{A}(X^{k+1})-b\|^2-\tau_k\Tr(X^{\star})-\frac{\varepsilon_k}{2}\|Y^{k+1}\|.
\end{equation}
\end{lemma}
\begin{proof}
By Proposition \ref{prop}, we have $S^{k+1}Y^{k+1}=\frac{1}{2}\cdot\grad\,\Psi_k(Y^{k+1})$ and so it holds $\langle S^{k+1}Y^{k+1}, Y^{k+1}\rangle\le\frac{1}{2}\cdot\|\grad\,\Psi_k(Y^{k+1})\|\cdot\|Y^{k+1}\|$.
Thus from the stopping criteria \eqref{stopcri}, we obtain
\begin{equation}\label{sec5:eq2}
    \langle S^{k+1}, X^{\star}-X^{k+1}\rangle=\langle S^{k+1}, X^{\star}\rangle-\langle S^{k+1}Y^{k+1}, Y^{k+1}\rangle\ge -\tau_k\Tr(X^{\star})-\frac{\varepsilon_k}{2}\|Y^{k+1}\|.
\end{equation}
Substituting $C-\cA^*(y^k-\sigma_k(\mathcal{A}(X^{k+1})-b))-\cB^*(z^{k+1})$ for $S^{k+1}$ in \eqref{sec5:eq2} where $z^{k+1}$ is given by \eqref{multiplierX} with $X\coloneqq X^{k+1}$, we obtain
\begin{align}\label{sec5:eq3}
    &\langle C-\cA^*(y^k-\sigma_k(\mathcal{A}(X^{k+1})-b))-\cB^*(z^{k+1}), X^{\star}-X^{k+1}\rangle \notag\\
    =\,&\langle C, X^{\star}-X^{k+1}\rangle-\langle y^k-\sigma_k(\mathcal{A}(X^{k+1})-b), \cA(X^{\star})-\cA(X^{k+1})\rangle \notag\\ &\quad\quad\quad\quad\quad\quad\quad\quad\quad\quad\quad\quad\quad\quad\quad\quad\quad-\langle z^{k+1}, \cB(X^{\star})-\cB(X^{k+1})\rangle \notag\\
    =\,&\langle C, X^{\star}-X^{k+1}\rangle+\langle y^k-\sigma_k(\mathcal{A}(X^{k+1})-b), \cA(X^{k+1})-b\rangle \notag\\
    \ge\,& -\tau_k\Tr(X^{\star})-\frac{\varepsilon_k}{2}\|Y^{k+1}\|,
\end{align}
where the second equality is because $\cA(X^{\star})=b$ and $\cB(X^{\star})=\cB(X^{k+1})=d$.
On the other hand, we have
\begin{equation*}
\langle S^{\star},X^{k+1}-X^{\star}\rangle=\langle C-\cA^*(y^{\star})-\cB^*(z^{\star}), X^{k+1}-X^{\star}\rangle=\langle S^{\star},X^{k+1}\rangle\ge0,
\end{equation*}
which gives
\begin{equation}\label{sec5:eq4}
\langle C, X^{k+1}-X^{\star}\rangle-\langle y^{\star}, \cA(X^{k+1})-b\rangle\ge0.
\end{equation}
Summing \eqref{sec5:eq3} and \eqref{sec5:eq4} gives \eqref{sec5:eq5} as desired.
\end{proof}

\begin{lemma}\label{sec5:lm2}
Assume that $\{Y^{k}\}_{k\in\N}$ is bounded. The sequence $\{\cA(X^k)\}_{k\ge1}$ converges to $b$.
\end{lemma}
\begin{proof}
Let $(X^{\star},y^{\star},S^{\star})$ be a KKT point of \eqref{sdp1}. For all $k\ge1$, by invoking Lemma \ref{sec5:lm1}, we have
\begin{align*}
\|y^{k+1}-y^{\star}\|^2&=\|y^{k}-y^{\star}\|^2-2\sigma_k\langle y^{k}-y^{\star}, \cA(X^{k+1})-b\rangle+\sigma_k^2\|\cA(X^{k+1})-b\|^2\\
&\le\|y^{k}-y^{\star}\|^2+\sigma_k(2\tau_k\Tr(X^{\star})+\varepsilon_k\|Y^{k+1}\|)-\sigma_k^2\|\cA(X^{k+1})-b\|^2.
\end{align*}
For arbitrary $N\ge1$, summing the above inequality for $k=1,\ldots,N$, we have
\begin{align*}
0&\le\sum_{k=1}^N\left(\|y^{k}-y^{\star}\|^2-\|y^{k+1}-y^{\star}\|^2\right)-\sum_{k=1}^N\sigma_k^2\|\cA(X^{k+1})-b\|^2\\
&\quad\quad\quad\quad\quad\quad\quad\quad\quad\quad\quad\quad\quad\quad\quad\quad+\sum_{k=1}^N\sigma_k\left(2\tau_k\Tr(X^{\star})+\varepsilon_k\|Y^{k+1}\|\right)\\
&\le\|y^{1}-y^{\star}\|^2-\|y^{N+1}-y^{\star}\|^2-\sigma_{\min}^2\sum_{k=1}^N\|\cA(X^{k+1})-b\|^2\\
&\quad\quad\quad\quad\quad\quad\quad\quad\quad\quad\quad\quad\quad\quad\quad+2\sigma_{\max}\Tr(X^{\star})\sum_{k=1}^N\tau_k+\sigma_{\max}D\sum_{k=1}^N\varepsilon_k,
\end{align*}
where $D=\max\,\{\|Y^{k+1}\|:k\ge1\}$. It follows
\begin{equation}\label{sec5:eq6}
    \sigma_{\min}\sum_{k=1}^N\|\cA(X^{k+1})-b\|^2\le\|y^{1}-y^{\star}\|^2+2\sigma_{\max}\Tr(X^{\star})\sum_{k=1}^N\tau_k+\sigma_{\max}D\sum_{k=1}^N\varepsilon_k.
\end{equation}
Now because the right-hand side of \eqref{sec5:eq6} is bounded and $N$ is arbitrary, we see that $\cA(X^k)$ must converge to $b$.
\end{proof}

{\em Proof of Theorem \ref{gc-thm}.} Let $(X^{\star},y^{\star},S^{\star})$ be a KKT point of \eqref{sdp1}. Assume $\hat{X}=\lim_{i\rightarrow\infty}X^{k_i}$. By Lemma \ref{sec5:lm2}, $\hat{X}$ is feasible to \eqref{sdp1}. Substituting $k_i$ for $k$ in \eqref{sec5:eq3} and then letting $i\rightarrow\infty$, by virtue of Lemma \ref{sec5:lm2} we obtain $\langle C, X^{\star}-\hat{X}\rangle\ge0$, which implies $\langle C, \hat{X}\rangle\le\langle C, X^{\star}\rangle$. Therefore, $\hat{X}$ is an optimal solution of \eqref{sdp1}. Because $\lim_{k\rightarrow\infty}\tau_k=0$, we must have $\hat{S}\succeq0$. Moreover, noting $\langle S^k,X^k\rangle=|\langle S^kY^k,Y^k\rangle|\le\frac{\varepsilon_{k-1}}{2}\cdot\|Y^k\|$, since $\lim_{k\rightarrow\infty}\varepsilon_k=0$ and $\{Y^{k}\}_{k\in\N}$ is bounded, we have $\langle \hat{X},\hat{S}\rangle=0$. It follows that $(\hat{X},\hat{y},\hat{S})$ is a KKT point of \eqref{sdp1}.

\begin{remark}
It could be seen from the above proof that Theorem 4.9 remains valid with a fixed penalty parameter $\sigma$. However, in the next section we will describe an adaptive strategy for updating the penalty parameter in order to improve the performance of the algorithm.
\end{remark}

\begin{remark}
It is clear that if the manifold $N_p$ is bounded (e.g., a sphere or oblique manifold), then $\{Y^k\}$ must be bounded. For the general case, it was established in \cite{rockafellar1976augmented} that the dual iterate by the ALM coincides with the proximal update on the dual function by the proximal point method, from which the boundedness of the primal sequence $\{X^k\}$ (and hence $\{Y^k\}$) can be deduced under mild conditions (see Proposition 2(c) of \cite{liao2025inexact}).
\end{remark}

\section{The practical algorithm}\label{algorithms}
Before giving the practical algorithm, we first elaborate the strategies for adjusting the factorization size $p$ and the penalty parameter $\sigma$ which are crucial to enhancing the practical performance of the algorithm.

\subsection{Dynamically adjusting the factorization size $p$}\label{sec5-1}
The computational complexity of the ALM subproblem \eqref{subpY} heavily depends on the factorization size $p$. To minimize the computational burden, we propose an effective strategy for dynamically adjusting the value of $p$ inspired by \cite{wang2023decomposition}. The recipe behind the strategy is based on two key ingredients: (1) increasing $p$ to escape from saddle points as discussed in Section \ref{sec5-1}; (2) decreasing $p$ by estimating ranks so that the computational burden of the ALM subproblem is as low as possible.
More specifically, suppose that $\{s_i\}_{i}$ are the singular values of $Y$ sorted from large to small. Then the rank of $Y$ is estimated by
\begin{equation}
    r = \argmax_i\,\{i\mid s_i>\theta s_1\},
\end{equation}
provided some threshold $\theta\in(0,1)$. Once the estimated rank $r$ of $Y$ is determined, we are able to construct a rank-$r$ approximation of $Y$ as follows. Suppose that $Y$ has the singular value decomposition 
\begin{equation*}
    Y=WDV^{\intercal},
\end{equation*}
where the diagonal of $D$ is sorted from large to small. We then take 
\begin{equation*}
    Y'=W_rD_r,
\end{equation*}
as a rank-$r$ approximation of $Y$, where $W_r$ is the submatrix of $W$ consisting of the first $r$ columns and $D_r$ is the upper-left $r\times r$ submatrix of $D$. Let 
\begin{equation*}
    \delta=\min\,\{n_{\rm{ne}},\delta_{\rm{ne}}\},
\end{equation*}
where $n_{\rm{ne}}$ is the number of negative eigenvalues of $S$ and $\delta_{\rm{ne}}\in\N$ is a tunable parameter. We then update the factorization size $p$ by letting $p=r+\delta$ and accordingly let $Y=[Y', 0_{n\times\delta}]$. To obtain a descent direction, let $U=[0_{n\times r}, v_1, \cdots, v_{\delta}]$ where $v_1,\ldots,v_{\delta}$ are the eigenvectors corresponding to the $\delta$ smallest eigenvalues of $S$. To summarize, the size updating strategy operates through dual complementary mechanisms: (1) size reduction via truncated singular value decomposition on the matrix $Y$ (this makes $Y$ be of full rank); (2) size expansion when $S\nsucceq0$ in order to escape from saddle points. Therefore, when $S\nsucceq0$, we can compute a descent direction via size expansion; when $S\succeq0$, the ALM iterations drive $\|\mathcal{A}(X)-b\|\to0$ and $XS\to0$ so that the converging point $Y^*$ provides an optimal solution ($X^*=Y^*(Y^*)^{\intercal}$) of the original SDP.

\begin{remark}
A small $\delta_{\rm{ne}}$ typically makes the ALM to converge slowly whereas a large $\delta_{\rm{ne}}$ makes the factorization size $p$ to grow rapidly. Therefore, the value of the parameter $\delta_{\rm{ne}}$ should be chosen to balance these two aspects.
\end{remark}

\begin{remark}
The above strategy for adjusting the factorization size is adapted from \cite{wang2023decomposition} with two distinctions: (1) we rely on a different procedure to estimate ranks; (2) we introduce the parameter $\delta_{\rm{ne}}$ to control the maximum increment of the factorization size at each step.
\end{remark}

\subsection{Self-adaptively updating the penalty parameter $\sigma$}\label{sec5-2}
Now we describe the strategy of self-adaptively updating the penalty parameter $\sigma$. Unlike usual ALMs using a monotonically nondecreasing sequence of penalty parameters, our strategy allows one to self-adaptively increase or decrease the penalty parameter. More concretely, we propose the following updating rules:

\begin{equation}\label{sec5:eq1}
    \sigma_{k+1}=\begin{cases}
    \min\,\{\gamma\sigma_k,\sigma_{\max}\},&\text{ if }\|\mathcal{A}(X^{k+1})-b\|/(1 +\left\lVert b \right\rVert)>\tau\|\grad\,\Psi_k(Y^{k+1})\|,\\
    \max\,\{\sigma_k/\gamma,\sigma_{\min}\},&\text{ otherwise,}
    \end{cases}
\end{equation}
where $\gamma>1,\tau>0,\sigma_{\min},\sigma_{\max}>0$ are constants. The intuition behind \eqref{sec5:eq1} is the following: the inequality $\|\mathcal{A}(X^{k+1})-b\|/(1 +\left\lVert b \right\rVert)>\tau\|\grad\,\Psi_k(Y^{k+1})\|$ indicates that the progress of feasibility is not satisfactory and hence we increase the penalty parameter by setting $\sigma_{k+1}=\min\,\{\gamma\sigma_k,\sigma_{\max}\}$; otherwise, the progress of feasibility is satisfactory and we may decrease the penalty parameter by setting $\sigma_{k+1}=\max\,\{\sigma_k/\gamma,\sigma_{\min}\}$. In doing so, the penalty parameter will not become too large through the iterations of the algorithm. We point out that a large penalty parameter makes the ALM subproblem \eqref{subpY} more difficult to solve, and thus preventing the penalty parameter from becoming large would improve the performance of the algorithm. 

\subsection{The algorithm}
Our practical algorithm is presented below in Algorithm \ref{alg1}.

\begin{algorithm}
\renewcommand{\algorithmicrequire}{\textbf{Input:}}
\renewcommand{\algorithmicensure}{\textbf{Output:}}
\caption{ManiSDP}\label{alg1}
	\begin{algorithmic}[1]
   	\REQUIRE $\cA,b,\cB,d,C,\sigma_0>0,\sigma_{\min}>0$,$\sigma_{\max}>0$, $p_0\in\N$, $\gamma>1,\tau>0$, $\delta_{\rm{ne}}\in\N$, $\theta\in(0,1)$
   	\ENSURE $(X,y,S)$
		\STATE $k \gets 0$, $y^0 \gets 0$, $p\gets p_0$, $Y^0 \gets 0_{n\times p}$, $U \gets 0_{n\times p}$;
		\WHILE{stopping criteria do not fulfill}
		\STATE Solve the ALM subproblem \eqref{subpY} inexactly with $U$ to obtain an approximate minimizer $Y^{k+1}$;
		\STATE $X^{k+1} \gets Y^{k+1}(Y^{k+1})^{\intercal}$, $y^{k+1} \gets y^k-\sigma_k(\mathcal{A}(X^{k+1})-b)$;
            \STATE $z_i\gets \frac{\Tr\left(B_i\nabla\Phi_k\left(X^{k+1}\right)X^{k+1}\right)}{\Tr\left(B_i^2X^{k+1}\right)}$, $i=1,\ldots,l$;
            \STATE $S^{k+1} \gets C-\mathcal{A}^*(y^{k+1})-\mathcal{B}^*(z)$;
		\STATE Compute a descent direction $U$ from $S^{k+1}$ according to Theorem \ref{sec5:thm1};
		\STATE Update $p$ as described in Section \ref{sec5-1};
            \IF{$\|\mathcal{A}(X^{k+1})-b\|/(1 +\left\lVert b \right\rVert)>\tau\|\grad\,\Psi_k(Y^{k+1})\|$}
            \STATE $\sigma_{k+1} \gets \min\,\{\gamma\sigma_k,\sigma_{\max}\}$;
            \ELSE
            \STATE $\sigma_{k+1} \gets \max\,\{\sigma_k/\gamma,\sigma_{\min}\}$;
            \ENDIF
            \STATE $k \gets k+1$;
		\ENDWHILE
		\RETURN $(X^{k},y^{k},S^{k})$;
	\end{algorithmic}
\end{algorithm}

We now make a few remarks on Algorithm \ref{alg1}.
\begin{itemize}
    \item The initial value $p_0$ of the factorization size $p$ is typically set to $1$ or $2$. However, for large-scale SDPs, setting a larger $p_0$ could be more advantageous. The value of $\gamma$ is typically set to $2$. The optimal setting of the other parameters is highly problem-dependent.
    \item The values of $y^0$ and/or $Y^0$ can be provided according to some initial guess of optimal solutions for warm-starting which may further improve the performance of the algorithm.
    \item At Step 3, the ALM subproblem is inexactly solved by performing a fixed number of iterations with the Riemannian Trust-Region method along the descent direction $U$ computed at Step 7. So it is not guaranteed that the approximate minimizer of the ALM subproblem at step 3 would fulfill the stopping criteria \eqref{stopcri}. In other words, we solve the ALM subproblem without imposing optimality requirements. Intriguingly, global convergence is still observed in all numerical experiments presented in Section \ref{experiments}.
    \item Most computation of the algorithm could be performed with $Y^{k+1}$ to avoid forming the big matrix $X^{k+1}=Y^{k+1}(Y^{k+1})^{\intercal}$.
    \item At Step 7, we need to perform an eigenvalue decomposition at each outer iteration of the ALM. As mentioned earlier, the ALM can enjoy fast linear convergence when the Riemannian trust-region method is adopted to solve the ALM subproblem. For the numerical examples tested in this paper, we perform \emph{full} eigenvalue decomposition and the ALM typically returns an approximately optimal solution with KKT residues $<$1e-8 in a few tens of outer iterations, implying that only a few tens of eigenvalue decompositions are required. Therefore, the computational cost of (full) eigenvalue decomposition is manageable (at least for the SDP size investigated in this study). In its current form, the algorithm may not be suitable for large-scale problems where performing full eigenvalue decomposition is computationally prohibitive.
    Nonetheless, we could employ partial eigenvalue decomposition to further decrease the computational cost when tackling SDPs of larger size, and we leave its numerical implementation and detailed analysis in future work.
\end{itemize}

The following theorem provides a posterior guarantee of global optimality for the output of Algorithm \ref{alg1}.
\begin{theorem}\label{sec5:thm2}
If the sequence $\{(X^{k},y^{k},S^{k})\}_{k\in\N}$ generated by Algorithm \ref{alg1} converges so that $\lim_{k\rightarrow\infty}X^{k}=X^{\star}$, $\lim_{k\rightarrow\infty}y^{k}=y^{\star}$, and $\lim_{k\rightarrow\infty}S^{k}=S^{\star}$. Then $(X^{\star},y^{\star},S^{\star})$ is a KKT point of \eqref{sdp1}, i.e., it satisfies \eqref{optimality} with $z^{\star}$ being given by
\begin{equation*}
z^{\star}_i\coloneqq\frac{\Tr(B_i(C-\cA^*(y^{\star}))X^{\star})}{\Tr(B_i^2X^{\star})}\text{ for }i=1,\ldots,l.
\end{equation*}
\end{theorem}
\begin{proof}
By the updating rule $y^{k+1}=y^k-\sigma_k(\mathcal{A}(X^{k+1})-b)$ and the convergence of $\{y^{k}\}_{k\in\N}$, we must have $\cA(X^{\star})=b$ and thus $X^{\star}$ is a feasible solution to \eqref{sdp1}. If $X^{\star}$ is not optimal, then by Theorem \ref{sec4:thm3}, we have $S^{\star}\nsucceq0$. So there exists a second-order descent direction for \eqref{subpY} starting from $Y^{\star}$ due to Theorem \ref{sec5:thm1}, which contradicts to the fact that $X^{\star}$ is the limit point of $\{X^{k}\}_{k\in\N}$. Thus $X^{\star}$ is optimal and $S^{\star}\succeq0$. Moreover, as $S^{\star}$ is the dual variable to $X^{\star}$, we have $\langle X^{\star}, S^{\star}\rangle=0$ by strong duality.
\end{proof}

\section{Numerical experiments}\label{experiments}
In this section, we conduct comprehensive numerical experiments to benchmark our solver {\tt ManiSDP} which implements Algorithm \ref{alg1} in MATLAB. In particular, {\tt Manopt 7.1} \cite{manopt} is employed by {\tt ManiSDP} to solve the Riemannian manifold optimization problem \eqref{subpY}. {\tt ManiSDP} is freely available at 
\begin{center}
\url{https://github.com/wangjie212/ManiSDP-matlab}.
\end{center}

{\bf Hardware.} All numerical experiments were performed on a desktop computer with Intel(R) Core(TM) i9-10900 CPU@2.80GHz and 64G RAM.

{\bf Baseline Solvers.} We compare the performance of {\tt ManiSDP} with that of four advanced SDP solvers: {\tt MOSEK 10.0}, {\tt SDPLR 1.03}, 
{\tt SDPNAL+}, {\tt STRIDE}. We explain why to choose these four baseline solvers: {\tt MOSEK} is chosen as it is a representative interior-point solver; {\tt SDPLR} is chosen as it is a representative solver that also exploits the low-rank property via the Burer-Monteiro factorization; {\tt SDPNAL+} is chosen as it is a representative solver that combines first-order with second-order methods and is designed to solve large-scale SDPs; {\tt STRIDE} is chosen as it specializes to solve large-scale SDP relaxations arising from polynomial optimization problems and exploits the rank-one property. It would be very interesting to compare also with the solver {\tt SDPDAL} of \cite{wang2023decomposition}, which is, however, currently impossible as {\tt SDPDAL} is not publicly available.
In the following, running time of solvers is measured in seconds; ``-" indicates that the solver encounters an out of memory error; ``$*$" indicates that running time exceeds $10,000$s; ``$**$" indicates that the solver returns certain numerical error. 

{\bf Stopping Criteria.} To measure the feasibility and optimality of an approximate solution $(X, y, S)\in \S^+_n\times\R^{m}\times\S^+_n $, we define the following KKT residues:
\begin{equation}\label{eq:res}
\eta_p = \frac{\left\lVert\cA(X)-b\right\rVert}{1 +\left\lVert b \right\rVert},\ \eta_d = \frac{\max\,\{0,-\lambda_{\min}(S)\}}{1+|\lambda_{\max}(S)|},\ \eta_g = \frac{|\left\langle C, X \right \rangle - b^{\intercal}y|}{1 + |\left\langle C, X \right \rangle| + |b^{\intercal}y|}.
\end{equation}
Given a tolerance $\text{tol}>0$, the SDP solver terminates when $\eta_{\max}\coloneqq\max\,\{\eta_p,\eta_d,\eta_g\}\leq\text{tol}$, and we set $\text{tol}=1\text{e-}8$ for all our experiments.

{\bf Benchmark Problems.}
To benchmark the solvers, we solve six classes of SDPs arising from different situations (the Max-Cut problem, the matrix completion problem, binary quadratic programs, minimizing quartic polynomials on the unit sphere, the robust rotation search problem, nearest structured rank deficient matrices), with a focus on second-order SDP relaxations for polynomial optimization problems as they are highly degenerate and are challenging for most SDP solvers.

\subsection{The Max-Cut problem}
The Max-Cut problem is one of the basic combinatorial optimization problems, which is known to be NP-complete. Suppose that $G(V, E)$ is an undirected graph with nodes $V=\{1,\ldots,N\}$ and with edge weights $w_{ij}=w_{ji}$ for $\{i,j\}\in E$. Then the Max-Cut problem for $G$, aiming to find the maximum cut, can be formulated as the following binary quadratic program:
\begin{equation}\label{maxcut}
    \begin{cases}
    \max&\frac{1}{2}\sum_{\{i,j\}\in E}w_{ij}(1-x_ix_j)\\
    \text{s.t.}&1-x_i^2=0,\quad i=1,\ldots,N.
    \end{cases}\tag{Max-Cut}
\end{equation}

To provide an upper bound on the maximum cut, we can consider the following SDP relaxation for \eqref{maxcut}:
\begin{equation}\label{maxcut2}
    \begin{cases}
    \max&\frac{1}{4}\langle L,X\rangle\\
    \text{s.t.}&X_{ii}=1,\quad i=1,\ldots,N,\\
    &X\succeq0,
    \end{cases}
\end{equation}
where $L\in\S_N$ is the Laplacian matrix of $G$, defined by
\begin{equation*}
    L_{ij}=\begin{cases}-w_{ij},&\text{if }\{i,j\}\in E,\\
    \sum_{k}w_{ik},&\text{if }i=j,\\
    0,&\text{otherwise}.
    \end{cases}
\end{equation*}
Note that \eqref{maxcut2} fits in \eqref{sdp1} with $m=0$ and $\cM$ matching \eqref{maniX2}.

We select test graphs from the webpage \url{https://web.stanford.edu/~yyye/yyye/Gset/} with $N$ varying from $800$ to $20000$.
For each instance, we solve \eqref{maxcut2} using the solvers {\tt MOSEK}, {\tt SDPLR}, {\tt SDPNAL+}, and {\tt ManiSDP}, respectively.
The results are presented in Table \ref{sec6:table:maxcut}. The following conclusions can be drawn from the table. (i) {\tt MOSEK} can solve the instances with $N\le10000$ to high accuracy, but the running time significantly grows as $N$ increases. When $N=20000$, {\tt MOSEK} runs out of space due to the large memory consumption of interior point methods. (ii) {\tt SDPNAL+} is very inefficient in solving this type of SDPs. For instance, when $N\ge3000$, {\tt SDPNAL+} needs over $10000$s to output the final result. (iii) Both {\tt SDPLR} and {\tt ManiSDP} can solve all instances to high accuracy, while {\tt ManiSDP} is even more accurate and is faster than {\tt SDPLR} by a factor of $2\sim10$. We refer the reader to \cite{boumal2016non,wang2023decomposition} for similar experiments on these graphs.

\begin{table}[htbp]
	\caption{Results for the Max-Cut problem.}\label{sec6:table:maxcut}
	\renewcommand\arraystretch{1.2}
	\centering
        \resizebox{\linewidth}{!}{
	\begin{tabular}{c|c|c|c|c|c|c|c|c|c}
		\multirow{2}{*}{graph}&\multirow{2}{*}{$N$}&\multicolumn{2}{c|}{{\tt MOSEK 10.0}}&\multicolumn{2}{c|}{{\tt SDPLR 1.03}}&\multicolumn{2}{c|}{{\tt SDPNAL+}}&\multicolumn{2}{c}{{\tt ManiSDP}}\\
		\cline{3-10}	&&$\eta_{\max}$&time&$\eta_{\max}$&time&$\eta_{\max}$&time&$\eta_{\max}$&time\\
		\hline
		G1&800&2.2e-09&3.51&7.4e-08&5.17&2.1e-09&52.2&1.6e-11&{\bf 0.54}\\
            \hline
		G2&800&2.7e-09&3.54&2.1e-07&3.29&3.5e-09&52.5&5.3e-14&{\bf 0.79}\\
            \hline
		G3&800&5.0e-09&3.50&2.1e-07&4.15&4.5e-09&38.6&4.3e-13&{\bf 0.81}\\
            \hline
		G4&800&2.6e-09&3.41&2.3e-07&3.14&2.4e-09&37.9&9.8e-13&{\bf 0.58}\\
            \hline
		G22&2000&1.0e-09&49.1&8.2e-08&12.5&8.6e-09&818&8.0e-12&{\bf 1.48}\\
            \hline
		G23&2000&1.7e-09&51.3&2.7e-07&22.4&2.3e-08&555&5.8e-12&{\bf 2.09}\\
            \hline
		G24&2000&9.2e-10&49.4&1.4e-06&6.95&4.3e-09&721&3.7e-12&{\bf 1.36}\\
            \hline
		G25&2000&1.3e-09&53.8&3.6e-07&10.4&2.8e-09&770&1.9e-12&{\bf 1.43}\\
            \hline
		G32&2000&2.3e-09&45.6&1.2e-07&22.9&7.0e-07&6463&1.6e-09&{\bf 4.34}\\
            \hline
		G43&1000&4.0e-09&6.31&7.3e-08&2.47&3.0e-09&59.1&2.7e-13&{\bf 0.68}\\
            \hline
            G44&1000&5.0e-09&6.27&3.1e-07&2.79&1.8e-08&61.2&4.2e-13&{\bf 0.62}\\
            \hline
            G45&1000&1.1e-09&6.34&2.0e-07&2.73&1.6e-08&61.6&1.5e-12&{\bf 0.59}\\
            \hline
		G48&3000&2.2e-09&108&1.1e-08&3.99&$*$&$*$&1.8e-17&{\bf 1.81}\\
            \hline
            G49&3000&3.0e-10&100&4.6e-08&4.21&$*$&$*$&1.1e-16&{\bf 1.94}\\
            \hline
            G50&3000&3.9e-14&112&3.0e-08&6.03&$*$&$*$&1.2e-14&{\bf 2.25}\\
            \hline
		G55&5000&3.7e-09&963&2.3e-07&34.5&$*$&$*$&6.6e-12&{\bf 17.7}\\
            \hline
            G56&5000&1.5e-09&847&7.5e-08&23.1&$*$&$*$&2.6e-12&{\bf 15.4}\\
            \hline
            G57&5000&7.0e-10&877&1.5e-07&120&$*$&$*$&5.0e-09&{\bf 30.8}\\
            \hline
            G58&5000&7.7e-10&1081&1.7e-07&101&$*$&$*$&1.2e-10&{\bf 29.7}\\
            \hline
            G59&5000&4.5e-09&931&8.3e-08&63.6&$*$&$*$&2.3e-13&{\bf 30.3}\\
            \hline
		G60&7000&1.1e-09&2722&7.5e-08&67.6&$*$&$*$&4.5e-12&{\bf 35.9}\\
            \hline
            G61&7000&4.6e-10&2735&1.0e-07&114&$*$&$*$&4.4e-10&{\bf 47.0}\\
            \hline
            G62&7000&2.2e-09&2484&3.1e-08&333&$*$&$*$&5.1e-09&{\bf 124}\\
            \hline
            G63&7000&2.6e-09&2978&3.2e-07&224&$*$&$*$&9.6e-09&{\bf 49.7}\\
            \hline
            G64&7000&1.2e-09&2886&5.8e-08&236&$*$&$*$&6.2e-09&{\bf 51.9}\\
            \hline
		G65&8000&1.5e-09&3794&4.6e-08&307&$*$&$*$&1.2e-09&{\bf 127}\\
            \hline
		G66&9000&2.6e-09&5464&5.7e-08&386&$*$&$*$&4.3e-09&{\bf 169}\\
            \hline
		G67&10000&4.6e-09&7363&4.1e-08&610&$*$&$*$&7.0e-09&{\bf 138}\\
            \hline
            G70&10000&4.1e-09&9451&3.1e-07&202&$*$&$*$&3.0e-12&{\bf 73.3}\\
            \hline
            G72&10000&9.7e-11&7728&1.5e-07&614&$*$&$*$&8.3e-09&{\bf 132}\\
            \hline
		G77&14000&$*$&$*$&6.9e-08&1177&$*$&$*$&1.9e-09&{\bf 452}\\
            \hline
		G81&20000&-&-&5.2e-08&3520&$*$&$*$&8.3e-09&{\bf 1934}\\
	\end{tabular}}
\end{table}

\subsection{The matrix completion problem}\label{mc}
The matrix completion problem seeks to recover a low-rank matrix $M\in\R^{s\times t}$ from a subset of entries $\{M_{ij}\}_{(i,j)\in\Omega}$. This can be formulized as the convex optimization problem:
\begin{equation}\label{sec6-eq6}
\begin{cases}
\inf\limits_{Z\in\R^{s\times t}}&\|Z\|_*\\
\,\,\,\,\,\,\rm{s.t.}&Z_{ij}=M_{ij},\quad\forall(i,j)\in\Omega,
\end{cases}\tag{MC}
\end{equation}
where $\|Z\|_*\coloneqq\Tr(Z^{\intercal}Z)^{\frac{1}{2}}$ is the nuclear norm of $Z$. Note that \eqref{sec6-eq6} can be equivalently cast as an SDP of size $(n,m)=(s+t,|\Omega|)$:
\begin{equation}\label{sec6-eq7}
\begin{cases}
\inf\limits_{X\in\S_n}&\Tr(X)\\
\,\,\,\rm{s.t.}&\left\langle\begin{bmatrix}0_{s\times s}&E_{ij}^{\intercal}\\ E_{ij}&0_{t\times t}\end{bmatrix},X\right\rangle=2M_{ij},\quad\forall(i,j)\in\Omega,\\
&X=\begin{bmatrix}U&Z^{\intercal}\\ Z&V\end{bmatrix}\succeq0,
\end{cases}
\end{equation}
where $E_{ij}$ is a $s\times t$ matrix with 1 at its $(i,j)$-position and 0 otherwise. Note also that \eqref{sec6-eq7} fits in \eqref{sdp1} with $\cM$ being the Euclidean manifold ($l=0$).
A famous result by Candes and Recht \cite{candes2012exact}, later improved by Candes and Tao \cite{candes2010power} states that, when $M$ is low-rank and incoherent, and the number of samples satisfies $|\Omega|\ge Cn(\log n)^2$ with some constant $C$, then $M$ can be exactly recovered by solving \eqref{sec6-eq7}. In this subsection, we consider random instances of the matrix completion problem \eqref{sec6-eq6}. To this end, we select $\Omega\subseteq[s]\times[t]$ uniformly at random from all subsets with cardinality $m$, and set $M=M_1M_2^{\intercal}$, where the entries of $M_1\in\R^{s\times k}$ and $M_2\in\R^{t\times k}$ are selected i.i.d. from the standard normal distribution. Here, we set $k=10$, $m=400n$ and take $s=t=1000,1500,2000,2500,3000,4000,5000,6000$ respectively to generate test instances.

For each instance, we solve \eqref{sec6-eq7} using the solvers {\tt MOSEK}, {\tt SDPLR}, {\tt SDPNAL+}, and {\tt ManiSDP}, respectively.
The results are presented in Table \ref{sec6:table:mc}, from which we make the following observations. (i) {\tt MOSEK} cannot solve any instance due to lack of enough memory. (ii) {\tt ManiSDP} is not only the most efficient but also the most accurate among the remaining three solvers. In particular, {\tt ManiSDP} is twice faster than {\tt SDPLR}, and is faster than {\tt SDPNAL+} by a order of magnitude.

\begin{table}[htbp]
	\caption{Results for the matrix completion problem.}\label{sec6:table:mc}
	\renewcommand\arraystretch{1.2}
	\centering
        \resizebox{\linewidth}{!}{
	\begin{tabular}{c|c|c|c|c|c|c|c|c|c|c}
		\multirow{2}{*}{$n$}&\multirow{2}{*}{trial}&\multirow{2}{*}{$m$}&\multicolumn{2}{c|}{{\tt MOSEK 10.0}}&\multicolumn{2}{c|}{{\tt SDPLR 1.03}}&\multicolumn{2}{c|}{{\tt SDPNAL+}}&\multicolumn{2}{c}{{\tt ManiSDP}}\\
		\cline{4-11}	&&&$\eta_{\max}$&time&$\eta_{\max}$&time&$\eta_{\max}$&time&$\eta_{\max}$&time\\
		\hline
		\multirow{3}{*}{2000}&\#1&550,536&-&-&1.7e-06&15.1&1.1e-08&69.9&5.3e-09&{\bf 7.87}\\
            &\#2&550,565&-&-&1.3e-06&14.7&4.5e-09&131&3.2e-10&{\bf 7.92}\\
            &\#3&550,590&-&-&8.6e-07&15.3&4.1e-09&143&3.7e-10&{\bf 8.32}\\
            \hline
		\multirow{3}{*}{3000}&\#1&930,328&-&-&1.8e-06&51.4&3.1e-08&238&9.1e-11&{\bf 21.6}\\
            &\#2&929,882&-&-&6.7e-07&49.4&3.2e-08&217&1.0e-10&{\bf 22.5}\\
            &\#3&930,080&-&-&3.6e-06&45.4&3.1e-08&216&4.1e-10&{\bf 22.2}\\
            \hline
		\multirow{3}{*}{4000}&\#1&1,318,563&-&-&1.0e-06&88.7&4.8e-08&532&3.2e-10&{\bf 48.3}\\
            &\#2&1,318,488&-&-&1.6e-06&99.1&1.9e-09&548&2.9e-10&{\bf 47.2}\\
            &\#3&1,318,885&-&-&2.2e-06&96.8&4.7e-08&519&2.7e-10&{\bf 49.5}\\
            \hline
            \multirow{3}{*}{5000}&\#1&1,711,980&-&-&1.2e-06&157&1.4e-09&1143&1.5e-10&{\bf 86.3}\\
            &\#2&1,711,445&-&-&1.0e-06&166&1.6e-09&1084&2.4e-10&{\bf 86.8}\\
            &\#3&1,711,660&-&-&1.1e-06&177&1.3e-09&1111&1.7e-10&{\bf 90.4}\\
            \hline
            \multirow{3}{*}{6000}&\#1&2,107,303&-&-&2.2e-07&272&2.1e-09&1883&4.7e-09&{\bf 139}\\
            &\#2&2,106,628&-&-&1.1e-06&260&2.2e-09&2001&1.5e-10&{\bf 145}\\
            &\#3&2,106,039&-&-&1.3e-06&271&2.5e-09&1979&2.0e-10&{\bf 145}\\
            \hline
            \multirow{3}{*}{8000}&\#1&2,900,179&-&-&2.1e-06&498&1.5e-08&3417&5.2e-11&{\bf 210}\\
            &\#2&2,900,585&-&-&3.4e-06&449&3.0e-08&4374&2.2e-10&{\bf 213}\\
            &\#3&2,900,182&-&-&3.2e-06&490&3.5e-08&4307&2.0e-10&{\bf 209}\\
            \hline
            \multirow{3}{*}{10000}&\#1&3,695,929&-&-&1.1e-06&800&1.4e-09&8370&1.9e-10&{\bf 369}\\
            &\#2&3,696,602&-&-&2.1e-06&789&8.6e-09&8849&2.1e-10&{\bf 363}\\
            &\#3&3,696,604&-&-&1.1e-06&798&7.2e-09&8502&2.5e-10&{\bf 354}\\
            \hline
            \multirow{3}{*}{12000}&\#1&4,493,420&-&-&7.8e-07&1310&$*$&$*$&8.3e-11&{\bf 568}\\
            &\#2&4,494,532&-&-&7.1e-07&1291&$*$&$*$&1.8e-10&{\bf 578}\\
            &\#3&4,493,391&-&-&3.5e-07&1330&$*$&$*$&4.9e-10&{\bf 590}\\
	\end{tabular}}
\end{table}

\subsection{Binary quadratic programs}
Let us consider the binary quadratic program given by
\begin{equation}\label{sec6:bqp}
\begin{cases}
\inf\limits_{\x\in\R^q} &\x^\intercal Q\x + \mathbf{c}^\intercal\x\\
\,\,\,\rm{s.t.}&x_i^2=1,\quad i=1,\ldots,q,\\
\end{cases}\tag{BQP}
\end{equation}
where $Q\in\S_q$ and $\mathbf{c}\in\R^{q}$. \eqref{sec6:bqp} includes the Max-Cut problem \eqref{maxcut} as well as many other combinatorial optimization problems as special cases. On the other hand, \eqref{sec6:bqp} belongs to the more general class of polynomial optimization problems whose objective functions and constraints are given by polynomials. For a polynomial optimization problem, there is a systematic way to construct a hierarchy of increasingly tighter SDP relaxations, known as the moment-SOS hierarchy or the Lasserre hierarchy\footnote{Under mild conditions, the optima of the hierarchy converge to the optimum of the polynomial optimization problem.} \cite{Las01}. The moment SDP relaxation arising from the Lasserre hierarchy typically admits low-rank optimal solutions. Interestingly, for the binary quadratic program \eqref{sec6:bqp}, the second-order moment relaxation is empirically tight on randomly generated instances as observed in \cite{lasserre2001explicit,yang2022inexact}. In the following we outline the ingredients of the second-order moment relaxation for \eqref{sec6:bqp}.
Let
\begin{equation*}
v(\x)\coloneqq[1,x_1,\ldots,x_q,x_1x_2,x_1x_3,\ldots,x_{q-1}x_q]^{\intercal}
\end{equation*}
be the vector of monomials in $\x$ up to degree two (excluding $x_i^2, i=1,\ldots,q$) and $M\coloneqq v(\x)v(\x)^{\intercal}$
be the corresponding \emph{moment matrix}. Then the objective function of \eqref{sec6:bqp} can be linearly expressed in terms of the entries of $M$. There are linear relationships among the entries of $M$ consisting of $M_{ij}=M_{kr}$ whenever $M_{ij}-M_{kr}$ is reduced to $0$ in the Gr\"obner basis $\{x_i^2-1\}_{i=1}^q$.
Let $\cA(X)=b$ collect all independent linear constraints obtained from these linear relationships when relaxing $M$ to an unknown PSD matrix $X$. Moreover, because of the constraints $x_i^2=1,i=1,\ldots,q$, the diagonal entries of $M$ are all ones and so we let $\cB(X)=d$ impose the unit-diagonal constraint on $X$. Consequently, we obtain the second-order moment relaxation for \eqref{sec6:bqp}, which fits in \eqref{sdp1} with $\cM$ matching \eqref{maniX2}.

For each $q\in\{10,20,30,40,50,60\}$, we generate three random instances of \eqref{sec6:bqp} by taking $Q\in\S_q$ with $Q_{ij}\sim\cN(0,1)$ and $\mathbf{c}\in\R^{q}$ with $c_{i}\sim\cN(0,1)$.
For each instance, we solve the second-order moment relaxation using the solvers {\tt MOSEK}, {\tt SDPLR}, {\tt SDPNAL+}, {\tt STRIDE} and {\tt ManiSDP}, respectively. The sizes of SDPs are recorded in Table \ref{sec6:table0:bpq} and
\begin{table}[htbp]
\caption{The sizes of SDPs for binary quadratic programs.}\label{sec6:table0:bpq}
\renewcommand\arraystretch{1.2}
\centering
\begin{tabular}{c|c|c|c|c|c|c}
$q$&10&20&30&40&50&60\\
\hline
$n$&56&211&466&821&1276&1831\\
\hline
$m$&1,256&16,361&77,316&236,121&564,776&1,155,281
\end{tabular}
\end{table}
the computational results are presented in Table \ref{sec6:table:bpq}. The following conclusions can be drawn from Table \ref{sec6:table:bpq}. (i) {\tt MOSEK} can solve small-scale instances ($q\le20$) to high accuracy, but the running time significantly grows as $q$ increases ($<1$s for $q=10$ while $\sim50$s for $q=20$). When $q\ge30$, {\tt MOSEK} runs out of space due to large memory consumption. (ii) {\tt SDPLR} can solve small/medium-scale instances ($q\le30$) to medium accuracy, but the running time significantly grows as $q$ increases. When $q\ge40$, {\tt SDPLR} needs over $10000$s to output the final result. (iii) {\tt SDPNAL+} can solve large-scale instances to medium/high accuracy, but the running time is pretty significant for large cases. (iv) Both {\tt STRIDE} and {\tt ManiSDP} can solve large-scale instances to high accuracy while {\tt ManiSDP} is faster than {\tt STRIDE} by a factor of $2\sim35$.

\begin{table}[htbp]
	\caption{Results for binary quadratic programs.}\label{sec6:table:bpq}
	\renewcommand\arraystretch{1.2}
	\centering
        \resizebox{\linewidth}{!}{
	\begin{tabular}{c|c|c|c|c|c|c|c|c|c|c|c}
		\multirow{2}{*}{$q$}&\multirow{2}{*}{trial}&\multicolumn{2}{c|}{{\tt MOSEK 10.0}}&\multicolumn{2}{c|}{{\tt SDPLR 1.03}}&\multicolumn{2}{c|}{{\tt SDPNAL+}}&\multicolumn{2}{c|}{{\tt STRIDE}}&\multicolumn{2}{c}{{\tt ManiSDP}}\\
		\cline{3-12}
		&&$\eta_{\max}$&time&$\eta_{\max}$&time&$\eta_{\max}$&time&$\eta_{\max}$&time&$\eta_{\max}$&time\\
		\hline
		\multirow{3}{*}{10}&\#1&2.6e-12&0.71&1.5e-06&0.52&1.9e-09&0.65&4.7e-13&0.79&3.9e-15&{\bf 0.14}\\
             &\#2&5.7e-14&0.84&6.0e-07&0.53&3.9e-09&1.37&3.4e-10&0.65&3.3e-15&{\bf 0.18}\\
             &\#3&8.0e-11&0.67&1.0e-06&1.27&1.5e-08&1.91&6.7e-13&0.68&4.2e-15&{\bf 0.29}\\
		\hline
		\multirow{3}{*}{20}&\#1&9.8e-10&49.0&3.9e-07&30.8&3.0e-09&28.8&7.4e-13&6.12&1.5e-14&{\bf 0.53}\\
             &\#2&9.0e-10&50.3&2.3e-08&113&1.7e-08&29.0&6.4e-13&6.98&1.3e-14&{\bf 0.61}\\
             &\#3&2.1e-12&47.9&6.6e-08&119&2.3e-07&12.5&2.9e-09&5.86&1.2e-14&{\bf 0.72}\\
		\hline
		\multirow{3}{*}{30}&\#1&-&-&2.1e-06&8384&1.7e-04&187&1.2e-12&65.4&2.8e-14&{\bf 3.93}\\
             &\#2&-&-&2.7e-07&2796&6.4e-09&95.5&3.1e-09&36.2&3.2e-14&{\bf 2.96}\\
             &\#3&-&-&1.6e-06&5698&7.8e-08&156&1.0e-12&60.3&2.8e-14&{\bf 4.01}\\
		\hline
		\multirow{3}{*}{40}&\#1&-&-&$*$&$*$&2.1e-08&813&4.4e-13&249&4.6e-14&{\bf 10.5}\\
             &\#2&-&-&$*$&$*$&1.3e-06&1514&8.5e-09&294&4.7e-14&{\bf 8.50}\\
             &\#3&-&-&$*$&$*$&1.3e-07&857&1.6e-12&321&4.4e-14&{\bf 10.0}\\
		\hline
		\multirow{3}{*}{50}&\#1&-&-&$*$&$*$&1.6e-07&3058&7.8e-09&826&6.4e-14&{\bf 31.1}\\
             &\#2&-&-&$*$&$*$&4.8e-08&6347&1.8e-12&1020&8.9e-14&{\bf 42.8}\\
             &\#3&-&-&$*$&$*$&7.0e-09&4800&8.2e-13&702&7.6e-14&{\bf 61.4}\\
		\hline
		\multirow{3}{*}{60}&\#1&-&-&$*$&$*$&$*$&$*$&1.3e-12&2118&9.4e-14&{\bf 94.3}\\
             &\#2&-&-&$*$&$*$&$*$&$*$&$*$&$*$&9.5e-14&{\bf 566}\\
             &\#3&-&-&$*$&$*$&$*$&$*$&3.3e-12&2704&8.7e-14&{\bf 150}\\
	\end{tabular}}
\end{table}

In Figures \ref{fig:1} and \ref{fig:2}, the factorization size and the maximal KKT residue per iteration in solving a random instance of \eqref{sec6:bqp} are shown for $q=10,20,30,40,50,60$, respectively.

\begin{figure}[htbp]	
\centering
\begin{tikzpicture}
\footnotesize
\scalefont{0.8} 
\begin{axis}[
sharp plot, 
xmode=normal,
width=12cm, height=6cm,  
xlabel= Iteration,
ylabel = Factorization size,
xlabel near ticks,
ylabel near ticks,
legend style={at={(0.15,0.5)},anchor=south},
]

\addplot[semithick,mark=*,mark options={scale=0.6}, color=lightgreen] coordinates { 
     (1,     2)
     (2,    10)
     (3,    18)
     (4,    14)
     (5,    13)
     (6,    10)
     (7,    12)
     (8,    12)
     (9,    10)
    (10,    10)
    (11,     5)
    (12,     2)
  };
\addlegendentry{$q=10$}

\addplot[semithick,mark=x,mark options={scale=0.6}, color=bordeaux] coordinates { 
     (1,     2)
     (2,    10)
     (3,    18)
     (4,    26)
     (5,    25)
     (6,    23)
     (7,    21)
     (8,    17)
     (9,    17)
    (10,    15)
    (11,    13)
    (12,     6)
    (13,     2)
};
\addlegendentry{$q=20$}

 \addplot[semithick,mark=+,mark options={scale=0.6}, color=color1] coordinates { 
(1,2)
(2,10)
(3,18)
(4,26)
(5,34)
(6,41)
(7,45)
(8,50)
(9,52)
(10,54)
(11,57)
(12,60)
(13,60)
(14,39)
(15,28)
(16,23)
(17,17)
(18,4)
(19,4)
(20,3)
(21,3)
(22,3)
(23,3)
(24,3)
(25,3)
(26,3)
(27,2)
};
\addlegendentry{$q=30$}

 \addplot[semithick,mark=*,mark options={scale=0.6}, color=color2] coordinates { 
     (1,    2)
     (2,    10)
     (3,    18)
     (4,    26)
     (5,    34)
     (6,    42)
     (7,    50)
     (8,    57)
     (9,    64)
    (10,    70)
    (11,    76)
    (12,    83)
    (13,    85)
    (14,    93)
    (15,    94)
    (16,    95)
    (17,    97)
    (18,    97)
    (19,    91)
    (20,    93)
    (21,    93)
    (22,    93)
    (23,    93)
    (24,    73)
    (25,    32)
    (26,     7)
    (27,     2)
};
    \addlegendentry{$q=40$}
    
\addplot[semithick,mark=x,mark options={scale=0.6}, color=color3] coordinates {
     (1,     2)
     (2,    10)
     (3,    18)
     (4,    26)
     (5,    34)
     (6,    42)
     (7,    50)
     (8,    58)
     (9,    66)
    (10,    74)
    (11,    82)
    (12,    89)
    (13,    95)
    (14,    97)
    (15,    97)
    (16,    98)
    (17,   100)
    (18,   101)
    (19,    97)
    (20,    92)
    (21,    77)
    (22,    31)
    (23,    22)
    (24,    12)
    (25,     6)
    (26,     3)
    (27,     3)
    (28,     2)
};
    \addlegendentry{$q=50$}

\addplot[semithick,mark=+,mark options={scale=0.6}, color=darkblue] coordinates { 
     (1,     2)
     (2,    10)
     (3,    18)
     (4,    26)
     (5,    34)
     (6,    42)
     (7,    50)
     (8,    58)
     (9,    66)
    (10,    74)
    (11,    82)
    (12,    90)
    (13,    98)
    (14,   106)
    (15,   114)
    (16,   122)
    (17,   129)
    (18,   137)
    (19,   145)
    (20,   151)
    (21,   158)
    (22,   162)
    (23,   167)
    (24,   169)
    (25,   174)
    (26,   178)
    (27,   183)
    (28,   184)
    (29,   183)
    (30,   182)
    (31,   172)
    (32,    58)
    (33,    18)
    (34,     8)
    (35,     2)
};
    \addlegendentry{$q=60$}
    
\end{axis}
\end{tikzpicture}
\caption{The factorization size per iteration in solving \eqref{sec6:bqp}.}
\label{fig:1}
\end{figure}
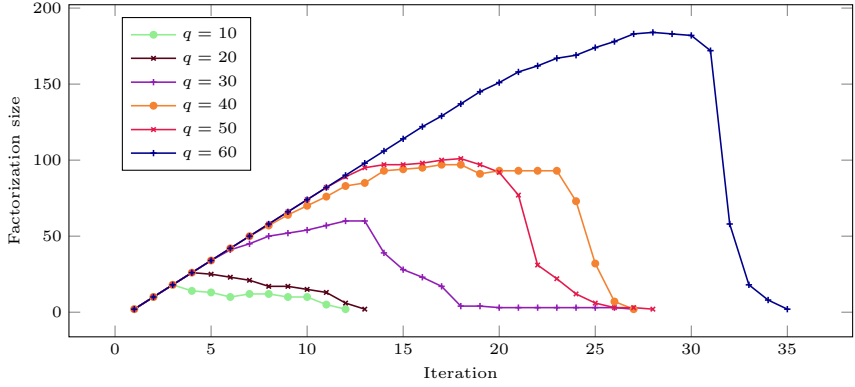  

\begin{figure}[htbp]	
\centering
\begin{tikzpicture}
\footnotesize
\scalefont{0.8} 
\begin{axis}[
sharp plot, 
xmode=normal,
width=12cm, height=6cm,  
xlabel= Iteration,
ylabel = $\log_{10}\eta_{\max}$,
xlabel near ticks,
ylabel near ticks,
legend style={at={(0.15,0.05)},anchor=south},
]

\addplot[semithick,mark=*,mark options={scale=0.6}, color=lightgreen] coordinates { 
    (1,1.3738)
    (2,1.3560)
    (3,1.1419)
    (4,0.9023)
    (5,0.6219)
    (6,0.3069)
    (7,-0.0956)
   (8,-0.6858)
   (9,-1.1237)
   (10,-2.9394)
   (11,-5.3559)
  (12,-14.4124)
  };
\addlegendentry{$q=10$}

\addplot[semithick,mark=x,mark options={scale=0.6}, color=bordeaux] coordinates { 
     (1,1.9135)
    (2,1.8965)
    (3,1.4612)
    (4,1.1012)
    (5,0.6259)
    (6,0.1540)
   (7,-0.4475)
   (8,-1.1710)
   (9,-1.9341)
   (10,-2.5508)
   (11,-3.3107)
   (12,-4.6742)
  (13,-13.8337)
};
\addlegendentry{$q=20$}

 \addplot[semithick,mark=+,mark options={scale=0.6}, color=color1] coordinates { 
    (1,2.2248)
    (2,2.1864)
    (3,1.7912)
    (4,1.3955)
    (5,1.0897)
    (6,0.8175)
    (7,0.5130)
    (8,0.2951)
    (9,0.5385)
    (10,0.4986)
    (11,0.0247)
   (12,-0.3354)
   (13,-0.4387)
   (14,-1.3433)
   (15,-1.7769)
   (16,-2.0898)
   (17,-1.8834)
   (18,-2.2084)
   (19,-2.1622)
   (20,-2.2713)
   (21,-2.6100)
   (22,-3.0742)
   (23,-3.4180)
   (24,-3.8291)
   (25,-4.3183)
   (26,-4.8165)
  (27,-13.5499)
};
\addlegendentry{$q=30$}

 \addplot[semithick,mark=*,mark options={scale=0.6}, color=color2] coordinates { 
     (1,2.4368)
    (2,2.4154)
    (3,2.0063)
    (4,1.7715)
    (5,1.4204)
    (6,1.1302)
    (7,0.9322)
    (8,1.1123)
    (9,0.9243)
    (10,1.0986)
    (11,0.8902)
    (12,1.1278)
    (13,0.7764)
    (14,0.9966)
    (15,0.7844)
    (16,1.0016)
    (17,0.9982)
    (18,0.6848)
    (19,0.3356)
    (20,0.0574)
    (21,0.0666)
   (22,-0.2671)
   (23,-0.4481)
   (24,-1.5928)
   (25,-2.7309)
   (26,-3.6376)
  (27,-13.3339)
};
    \addlegendentry{$q=40$}

\addplot[semithick,mark=x,mark options={scale=0.6}, color=color3] coordinates {
    (1,2.6183)
    (2,2.5571)
    (3,2.1762)
    (4,1.9270)
    (5,1.7303)
    (6,1.3707)
    (7,1.5407)
    (8,1.1199)
    (9,0.9821)
    (10,1.1742)
    (11,0.9429)
    (12,1.1421)
    (13,0.6878)
    (14,0.7042)
    (15,0.3370)
    (16,0.0363)
   (17,-0.0814)
   (18,-0.0843)
   (19,-0.4093)
   (20,-0.3858)
   (21,-0.0937)
   (22,-0.6127)
   (23,-0.6422)
   (24,-1.3879)
   (25,-2.2145)
   (26,-3.4785)
   (27,-4.8754)
  (28,-13.1952)
};
    \addlegendentry{$q=50$}

    \addplot[semithick,mark=+,mark options={scale=0.6}, color=darkblue] coordinates { 
     (1,2.7690)
    (2,2.7522)
    (3,2.4095)
    (4,2.2272)
    (5,2.0676)
    (6,1.8531)
    (7,1.5438)
    (8,1.6769)
    (9,1.4228)
    (10,1.2507)
    (11,1.6238)
    (12,1.3222)
    (13,1.5980)
    (14,1.7075)
    (15,1.3848)
    (16,1.1574)
    (17,1.3154)
    (18,1.1214)
    (19,1.3381)
    (20,0.9677)
    (21,1.1129)
    (22,0.7723)
    (23,0.5708)
    (24,0.7849)
    (25,0.3833)
    (26,0.3668)
    (27,0.4601)
    (28,0.0722)
   (29,-0.2549)
   (30,-0.3040)
   (31,-0.4775)
   (32,-1.6029)
   (33,-2.1874)
   (34,-3.0898)
  (35,-13.0250)
};
    \addlegendentry{$q=60$}
    
\end{axis}
\end{tikzpicture}
\caption{The maximal KKT residue per iteration in solving \eqref{sec6:bqp}.}
\label{fig:2}
\end{figure}
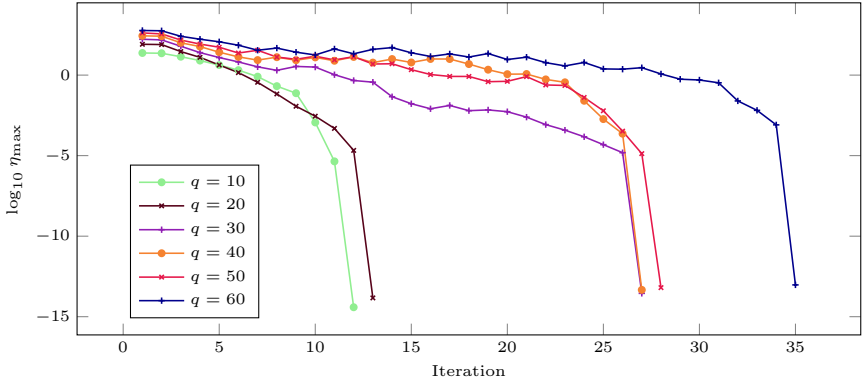  

To test the limit of {\tt ManiSDP}, we run {\tt ManiSDP} to solve the second-order moment relaxation of \eqref{sec6:bqp} with larger $q$. As shown in Table \ref{sec6:table1:bpq}, {\tt ManiSDP} can scale up to $q=120$ for which the SDP has matrix dimension $n=7261$ and contains $m=17,869,161$ affine constraints, far beyond the reach of other SDP solvers!

\begin{table}[htbp]
\caption{Results for large binary quadratic programs via {\tt ManiSDP}.}\label{sec6:table1:bpq}
\renewcommand\arraystretch{1.2}
\centering
\begin{tabular}{c|c|c|c|c|c|c}
$q$&70&80&90&100&110&120\\
\hline
$n$&2486&3241&4096&5051&6106&7261\\
\hline
$m$&2,119,636&3,589,841&5,717,896&8,675,801&12,655,556&17,869,161\\
\hline
$\eta_{\max}$&1.4e-13&1.7e-13&2.2e-13&2.5e-13&3.0e-13&3.5e-13\\
\hline
time&1050&1132&3279&5249&7053&30801
\end{tabular}
\end{table}

\subsection{Minimizing quartic polynomials on the unit sphere}
Let us consider the problem of minimizing a quartic polynomial on the unit sphere:
\begin{equation}\label{sec6:qs}
\begin{cases}
\inf\limits_{\x\in\R^q} &\mathbf{c}^\intercal\cdot[\x]_4\\
\,\,\,\rm{s.t.}&\sum_{i=1}^qx_i^2=1,\\
\end{cases}\tag{QS}
\end{equation}
where $[\x]_4$ is the vector of monomials in $\x$ up to degree four and $\mathbf{c}\in\R^{|[\x]_4|}$. As for \eqref{sec6:bqp}, the second-order moment relaxation is empirically tight on randomly generated instances of \eqref{sec6:qs} \cite{yang2022inexact}.
Let
\begin{equation*}
v(\x)\coloneqq[1,x_1,\ldots,x_q,x_1^2,x_1x_2,x_1x_3,\ldots,x_{q-1}x_q,x_q^2]^{\intercal}
\end{equation*}
be the vector of monomials in $\x$ up to degree two and $M\coloneqq v(\x)v(\x)^{\intercal}$
be the corresponding moment matrix. Then the objective function of \eqref{sec6:qs} can be linearly expressed in terms of the entries of $M$. There are linear relationships among the entries of $M$ consisting of all $M_{ij}=M_{kr}$. In addition, for each monomial $w\in v(\x)$, the constraint $\sum_{i=1}^qx_i^2=1$ gives $w(\sum_{i=1}^qx_i^2-1)=0$ which can be also linearly expressed in terms of the entries of $M$.
Let $\cA(X)=b$ collect all independent linear constraints obtained from these linear relationships when relaxing $M$ to an unknown PSD matrix $X$. We therefore obtain the second-order moment relaxation for \eqref{sec6:qs}, which fits in \eqref{sdp1} with $\cM$ being the Euclidean manifold ($l=0$).

For each $q\in\{10,20,30,40,50,60\}$, we generate three random instances of \eqref{sec6:qs} by taking $\mathbf{c}\in\R^{|[\x]_4|}$ with $c_{i}\sim\cN(0,1)$. For each instance, we solve the second-order moment relaxation using the solvers {\tt MOSEK}, {\tt SDPLR}, {\tt SDPNAL+}, {\tt STRIDE} and {\tt ManiSDP}, respectively.
The results are presented in Table \ref{sec6:table:qs}, from which we can draw the following conclusions. (i) {\tt MOSEK} can solve small-scale instances ($q\le20$) to high accuracy, but the running time significantly grows as $q$ increases ($<1$s for $q=10$ while $\sim50$s for $q=20$). When $q\ge30$, {\tt MOSEK} runs out of space due to large memory consumption. (ii) {\tt SDPLR} can solve all instances to medium accuracy, but the running time significantly grows as $q$ increases. (iii) {\tt ManiSDP} is the most efficient solver among the remaining three solvers. (iv) {\tt SDPNAL+} attains only medium accuracy for large-scale instances whereas {\tt STRIDE} and {\tt ManiSDP} can always attain high accuracy.

\begin{table}[htbp]
	\caption{Results for minimizing quartic polynomials on the unit sphere.}\label{sec6:table:qs}
	\renewcommand\arraystretch{1.2}
	\centering
        \resizebox{\linewidth}{!}{
	\begin{tabular}{c|c|c|c|c|c|c|c|c|c|c|c}
		\multirow{2}{*}{$q$}&\multirow{2}{*}{trial}&\multicolumn{2}{c|}{{\tt MOSEK 10.0}}&\multicolumn{2}{c|}{{\tt SDPLR 1.03}}&\multicolumn{2}{c|}{{\tt SDPNAL+}}&\multicolumn{2}{c|}{{\tt STRIDE}}&\multicolumn{2}{c}{{\tt ManiSDP}}\\
		\cline{3-12}
		&&$\eta_{\max}$&time&$\eta_{\max}$&time&$\eta_{\max}$&time&$\eta_{\max}$&time&$\eta_{\max}$&time\\
		\hline
		\multirow{3}{*}{10}&\#1&6.6e-11&0.79&2.0e-07&{\bf 0.07}&2.5e-09&0.45&3.9e-12&0.35&4.5e-09&0.18\\
             &\#2&8.8e-10&0.80&9.2e-07&0.56&1.3e-09&0.54&2.6e-12&0.52&7.0e-10&{\bf 0.30}\\
             &\#3&5.9e-10&0.79&9.0e-07&{\bf 0.04}&2.1e-09&0.54&2.3e-11&0.38&1.7e-10&0.18\\
             \hline
		\multirow{3}{*}{20}&\#1&7.4e-09&42.5&3.5e-06&1.35&1.2e-09&3.27&2.5e-12&2.74&3.7e-10&{\bf 0.95}\\
             &\#2&4.0e-10&49.3&3.7e-07&2.34&5.6e-09&3.36&4.7e-11&2.91&5.4e-10&{\bf 1.13}\\
             &\#3&1.0e-08&42.9&5.2e-08&1.10&9.9e-09&3.20&8.7e-13&3.01&1.1e-09&{\bf 0.82}\\
             \hline
		\multirow{3}{*}{30}&\#1&-&-&1.5e-06&20.8&1.2e-09&20.9&3.1e-11&18.6&1.7e-10&{\bf 6.46}\\
             &\#2&-&-&7.4e-07&38.4&1.5e-10&21.4&3.2e-13&19.4&8.4e-09&{\bf 5.63}\\
             &\#3&-&-&1.8e-07&19.6&1.1e-09&19.0&1.1e-12&22.3&3.7e-10&{\bf 5.92}\\
             \hline
		\multirow{3}{*}{40}&\#1&-&-&3.5e-07&689&1.0e-07&45.0&2.8e-13&42.8&4.3e-09&{\bf 28.7}\\
             &\#2&-&-&6.2e-07&272&1.4e-07&28.8&1.0e-12&39.1&4.2e-09&{\bf 19.4}\\
             &\#3&-&-&1.4e-07&261&3.9e-06&24.9&4.2e-11&39.3&8.9e-09&{\bf 20.1}\\
             \hline
		\multirow{3}{*}{50}&\#1&-&-&8.0e-07&1588&5.2e-07&68.9&2.3e-12&115&4.8e-09&{\bf 61.4}\\
             &\#2&-&-&8.2e-08&1183&2.9e-06&69.0&8.4e-11&105&2.6e-09&{\bf 49.9}\\
             &\#3&-&-&3.9e-07&2350&1.1e-06&71.4&5.1e-11&124&4.7e-09&{\bf 57.1}\\
		\hline
		\multirow{3}{*}{60}&\#1&-&-&2.4e-07&4167&5.7e-07&177&2.6e-12&194&3.9e-09&{\bf 109}\\
             &\#2&-&-&1.0e-08&7229&3.4e-07&237&3.6e-13&288&6.5e-10&{\bf 116}\\
             &\#3&-&-&3.6e-08&7752&4.7e-07&195&4.4e-13&209&2.1e-09&{\bf 173}\\
	\end{tabular}}
\end{table}

In Figures \ref{fig:5} and \ref{fig:6}, the factorization size and the maximal KKT residue per iteration in solving a random instance of \eqref{sec6:qs} are shown for $q=10,20,30,40,50,60$, respectively.

\begin{figure}[htbp]	
\centering
\begin{tikzpicture}
\footnotesize
\scalefont{0.8} 
\begin{axis}[
sharp plot, 
xmode=normal,
width=12cm, height=6cm,  
xlabel= Iteration,
ylabel = Factorization size,
xlabel near ticks,
ylabel near ticks,
legend style={at={(0.15,0.5)},anchor=south},
]

\addplot[semithick,mark=*,mark options={scale=0.6}, color=lightgreen] 
 file {plotdata/qs_fz_10.txt};
\addlegendentry{$q=10$}
 
\addplot[semithick,mark=x,mark options={scale=0.6}, color=bordeaux]
 file {plotdata/qs_fz_20.txt};
\addlegendentry{$q=20$}

 \addplot[semithick,mark=+,mark options={scale=0.6}, color=color1] 
 file {plotdata/qs_fz_30.txt};
\addlegendentry{$q=30$}

 \addplot[semithick,mark=*,mark options={scale=0.6}, color=color2]
 file {plotdata/qs_fz_40.txt};
    \addlegendentry{$q=40$}
    
\addplot[semithick,mark=x,mark options={scale=0.6}, color=color3] 
 file {plotdata/qs_fz_50.txt};
    \addlegendentry{$q=50$}

\addplot[semithick,mark=+,mark options={scale=0.6}, color=darkblue] 
 file {plotdata/qs_fz_60.txt};
    \addlegendentry{$q=60$}
    
\end{axis}
\end{tikzpicture}
\caption{The factorization size per iteration in solving \eqref{sec6:qs}.}
\label{fig:5}
\end{figure}
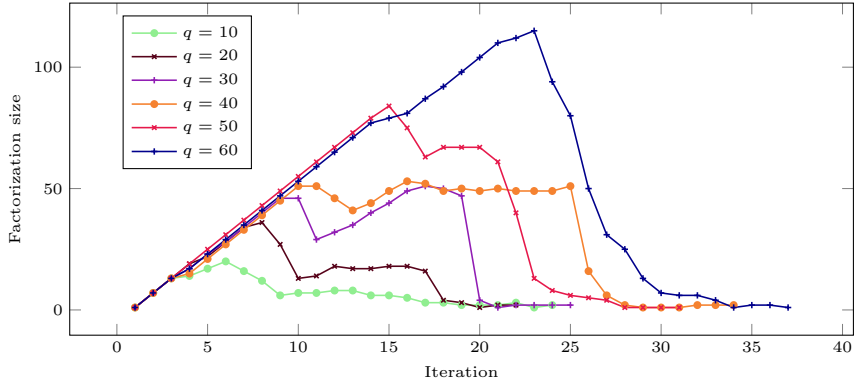  

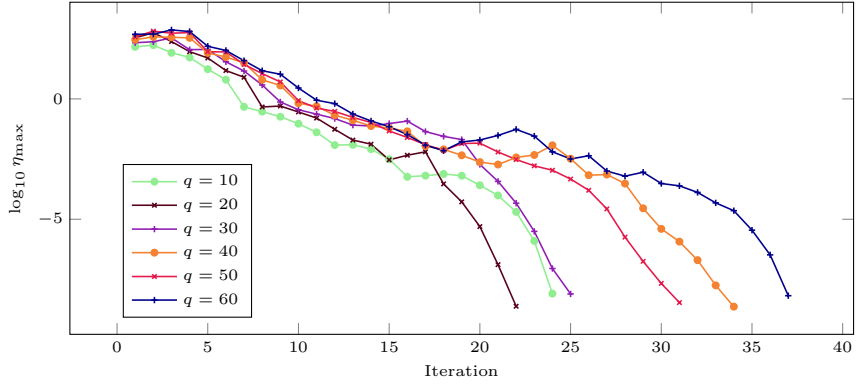
\begin{figure}[htbp]	
\centering
\begin{tikzpicture}
\footnotesize
\scalefont{0.8} 
\begin{axis}[
sharp plot, 
xmode=normal,
width=12cm, height=6cm,  
xlabel= Iteration,
ylabel = $\log_{10}\eta_{\max}$,
xlabel near ticks,
ylabel near ticks,
legend style={at={(0.15,0.05)},anchor=south},
]

\addplot[semithick,mark=*,mark options={scale=0.6}, color=lightgreen]
file {plotdata/qs_residue_10.txt};
\addlegendentry{$q=10$}

\addplot[semithick,mark=x,mark options={scale=0.6}, color=bordeaux]
file {plotdata/qs_residue_20.txt};
\addlegendentry{$q=20$}

 \addplot[semithick,mark=+,mark options={scale=0.6}, color=color1]
 file {plotdata/qs_residue_30.txt};
\addlegendentry{$q=30$}

 \addplot[semithick,mark=*,mark options={scale=0.6}, color=color2] 
 file {plotdata/qs_residue_40.txt};
    \addlegendentry{$q=40$}

\addplot[semithick,mark=x,mark options={scale=0.6}, color=color3] 
file {plotdata/qs_residue_50.txt};
    \addlegendentry{$q=50$}

    \addplot[semithick,mark=+,mark options={scale=0.6}, color=darkblue] 
    file {plotdata/qs_residue_60.txt};
    \addlegendentry{$q=60$}
    
\end{axis}
\end{tikzpicture}
\caption{The maximal KKT residue per iteration in solving \eqref{sec6:qs}.}
\label{fig:6}
\end{figure}  

\subsection{The robust rotation search problem}
The robust rotation search problem (also known as the Wahba problem with outliers) is to search for the best 3D rotation to align two sets of 3D points while explicitly tolerating outliers, which can be formulated as the nonlinear optimization problem:
\begin{equation}\label{eq:tlswahba}
  \min_{\|q\|=1} \sum_{i=1}^N \min \left\{\frac{\|\tilde{z}_i - q \circ \tilde{w}_i \circ q^{-1}\|^2 }{\beta_i^2}, 1 \right\},
\end{equation}
where $q$ is the unit quaternion parametrization of a 3D rotation, $(z_i\in\R^3,w_i\in\R^3)_{i=1}^N$ are given $N$ pairs of 3D points, $\tilde{z}\coloneqq[z^{\intercal},0]^{\intercal}\in\R^4$, $\tilde{w}\coloneqq[w^{\intercal},0]^{\intercal}\in\R^4$, $q^{-1}\coloneqq[-q_1,-q_2,-q_3,q_4]^{\intercal}$ is the inverse quaternion, ``$\circ$'' denotes the quaternion product, $\beta_i>0$ is a given threshold that determines the maximum inlier residual. Problem \eqref{eq:tlswahba} is a fundamental problem in aerospace, robotics and computer vision \cite{yang2019quaternion,yang2020teaser}. By introducing $N$ binary variables $\{\theta_i\}_{i=1}^N$, Problem \eqref{eq:tlswahba} can be equivalently reformulated as a polynomial optimization problem:
\begin{equation}\label{eq:wahba}
\min_{\substack{\|q\|=1,\\ \theta_i\in\{-1,1\},i=1,\dots,N}}\sum_{i=1}^N \frac{1+\theta_i}{2}\frac{\|\tilde{z}_i-q\circ\tilde{w}_i\circ q^{-1}\|^2}{\beta_i^2}+\frac{1-\theta_i}{2}. \tag{RRS}
\end{equation}
Each $\theta_i$ is used to decide whether the $i$-th pair of 3D points $(z_i,w_i)$ is an inlier or an outlier.

Yang and Carlone \cite{yang2019quaternion} proposed an SDP relaxation for \eqref{eq:wahba} that was empirically shown to be tight. Let $\x = [q^{\intercal},\theta_1,\dots,\theta_N]^{\intercal}\in \R^{N+4}$ be the decision variables of \eqref{eq:wahba}, and let 
\begin{equation}\label{eq:wahbaclone}
   v(\x) = [q^\intercal,\theta_1 q^\intercal, \dots, \theta_N q^\intercal]^\intercal \in \R^{4N+4} 
\end{equation}
be the sparse set of monomials in $\x$ of degree up to two. We build $M = v(\x)v(\x)^\intercal$ as the sparse moment matrix. Then the objective function of \eqref{eq:wahba} can be linearly expressed in terms of the entries of $M$. There are linear relationships among the entries of $M$: (1) the diagonal $4\times 4$ blocks of $M$ are all identical ($\theta_i^2 q q^\intercal = qq^\intercal$); (2) the off-diagonal $4\times 4$ blocks are symmetric ($\theta_i \theta_j q q^\intercal\in\S_4$). Let $\cA(X)=b$ collect all independent linear constraints obtained from these linear relationships when relaxing $M$ to an unknown PSD matrix $X$. In addition, because of the unit quaternion constraint, $M$ satisfies $\Tr(M) = N+1$ and so we let $\cB(X)=d$ impose the trace constraint on $X$. Consequently, this leads to an SDP relaxation of size
\begin{equation}
n = 4N+4, \quad m = 3N^2+13N+1,
\end{equation}
which fits in \eqref{sdp1} with $\cM$ matching \eqref{maniX1} after scaling $X$ by $\frac{1}{N+1}$.

For each $N\in\{50,100,150,200,300,500\}$, we generate three random instances of \eqref{eq:wahba}. For each instance, we solve the above SDP relaxation using the solvers {\tt MOSEK}, {\tt SDPLR}, {\tt SDPNAL+}, {\tt STRIDE} and {\tt ManiSDP}, respectively.
The results are presented in Table \ref{sec6:table:rs}. The following conclusions can be drawn from the table. (i) {\tt MOSEK} can solve small-scale instances ($N\le100$) to high accuracy, but the running time significantly grows as $N$ increases ($<20$s for $N=10$ while $\sim600$s for $N=100$). When $N\ge150$, {\tt MOSEK} runs out of space due to large memory consumption. (ii) Both {\tt SDPLR} and {\tt SDPNAL+} fail in solving these SDPs to even medium accuracy. (iii) Both {\tt STRIDE} and {\tt ManiSDP} can solve all instances to high accuracy while {\tt ManiSDP} is faster than {\tt STRIDE} by a factor of $3\sim10$.

\begin{table}[htbp]
	\caption{Results for the robust rotation search problem.}\label{sec6:table:rs}
	\renewcommand\arraystretch{1.2}
	\centering
        \resizebox{\linewidth}{!}{
	\begin{tabular}{c|c|c|c|c|c|c|c|c|c|c|c}
		\multirow{2}{*}{$N$}&\multirow{2}{*}{trial}&\multicolumn{2}{c|}{{\tt MOSEK 10.0}}&\multicolumn{2}{c|}{{\tt SDPLR 1.03}}&\multicolumn{2}{c|}{{\tt SDPNAL+}}&\multicolumn{2}{c|}{{\tt STRIDE}}&\multicolumn{2}{c}{{\tt ManiSDP}}\\
		\cline{3-12}	&&$\eta_{\max}$&time&$\eta_{\max}$&time&$\eta_{\max}$&time&$\eta_{\max}$&time&$\eta_{\max}$&time\\
		\hline
		\multirow{3}{*}{50}&\#1&4.7e-10&16.4&9.8e-03&12.5&1.1e-02&106&2.8e-09&18.3&6.6e-09&{\bf 3.02}\\
             &\#2&7.9e-10&19.3&3.1e-02&22.0&1.0e-02&96.3&7.3e-09&15.4&7.2e-10&{\bf 2.93}\\
             &\#3&1.1e-10&15.2&2.8e-03&19.4&1.1e-02&119&9.5e-09&15.4&5.5e-10&{\bf 3.59}\\
             \hline
		\multirow{3}{*}{100}&\#1&2.0e-11&622&3.6e-04&106&7.1e-02&642&3.1e-09&73.0&1.0e-09&{\bf 22.9}\\
             &\#2&1.8e-10&653&8.1e-04&78.1&3.8e-02&631&1.6e-09&67.4&3.6e-10&{\bf 20.3}\\
             &\#3&7.3e-12&590&2.9e-03&67.2&7.8e-02&597&4.8e-09&69.0&5.4e-10&{\bf 18.2}\\
             \hline
             \multirow{3}{*}{150}&\#1&-&-&2.0e-03&291&8.0e-02&1691&4.3e-11&249&1.6e-09&{\bf 33.5}\\
             &\#2&-&-&1.2e-03&233&6.4e-02&804&7.9e-09&171&4.7e-09&{\bf 36.5}\\
             &\#3&-&-&1.5e-01&416&1.2e-01&1491&9.9e-09&162&2.6e-09&{\bf 33.8}\\
             \hline
             \multirow{3}{*}{200}&\#1&-&-&3.1e-02&459&8.3e-02&2799&1.4e-09&254&6.3e-10&{\bf 65.3}\\
             &\#2&-&-&1.6e-01&761&6.5e-02&1653&2.9e-09&306&9.2e-10&{\bf 66.2}\\
             &\#3&-&-&3.8e-03&894&6.3e-02&2171&3.2e-11&220&8.5e-10&{\bf 67.9}\\
             \hline
             \multirow{3}{*}{300}&\#1&-&-&1.1e-03&1264&5.2e-02&3528&4.1e-10&1176&1.1e-09&{\bf 188}\\
             &\#2&-&-&7.3e-03&1787&4.9e-02&3421&8.0e-09&1458&3.6e-09&{\bf 190}\\
             &\#3&-&-&2.2e-03&1734&6.0e-02&4260&2.9e-09&868&1.2e-09&{\bf 203}\\
             \hline
             \multirow{3}{*}{500}&\#1&-&-&$*$&$*$&$*$&$*$&7.1e-09&5627&5.2e-10&{\bf 601}\\
             &\#2&-&-&5.4e-02&9574&$*$&$*$&4.5e-10&4884&1.9e-09&{\bf 801}\\
             &\#3&-&-&$*$&$*$&$*$&$*$&3.4e-09&7878&5.0e-09&{\bf 1055}\\
	\end{tabular}}
\end{table}

In Figures \ref{fig:3} and \ref{fig:4}, the factorization size and the maximal KKT residue per iteration in solving a random instance of \eqref{eq:wahba} are displayed for $N=50,100,150,200,300$, $500$, respectively.

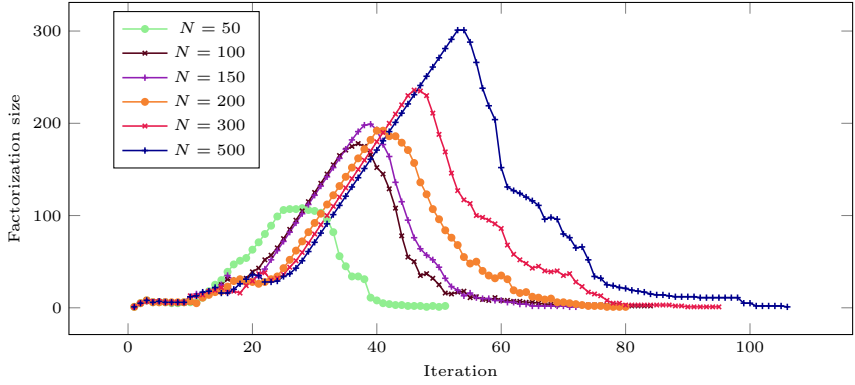
\begin{figure}[htbp]	
\centering
\begin{tikzpicture}
\footnotesize
\scalefont{0.8} 
\begin{axis}[
sharp plot, 
xmode=normal,
width=12cm, height=6cm,  
xlabel= Iteration,
ylabel = Factorization size,
xlabel near ticks,
ylabel near ticks,
legend style={at={(0.15,0.5)},anchor=south},
]

\addplot[semithick,mark=*,mark options={scale=0.6}, color=lightgreen] 
 file {plotdata/rs_fz_50.txt};
\addlegendentry{$N=50$}
 
\addplot[semithick,mark=x,mark options={scale=0.6}, color=bordeaux]
 file {plotdata/rs_fz_100.txt};
\addlegendentry{$N=100$}

 \addplot[semithick,mark=+,mark options={scale=0.6}, color=color1] 
 file {plotdata/rs_fz_150.txt};
\addlegendentry{$N=150$}

 \addplot[semithick,mark=*,mark options={scale=0.6}, color=color2]
 file {plotdata/rs_fz_200.txt};
    \addlegendentry{$N=200$}
    
\addplot[semithick,mark=x,mark options={scale=0.6}, color=color3] 
 file {plotdata/rs_fz_300.txt};
    \addlegendentry{$N=300$}

\addplot[semithick,mark=+,mark options={scale=0.6}, color=darkblue] 
 file {plotdata/rs_fz_500.txt};
    \addlegendentry{$N=500$}
    
\end{axis}
\end{tikzpicture}
\caption{The factorization size per iteration in solving \eqref{eq:wahba}.}
\label{fig:3}
\end{figure}  

\begin{figure}[htbp]	
\centering
\begin{tikzpicture}
\footnotesize
\scalefont{0.8} 
\begin{axis}[
sharp plot, 
xmode=normal,
width=12cm, height=6cm,  
xlabel= Iteration,
ylabel = $\log_{10}\eta_{\max}$,
xlabel near ticks,
ylabel near ticks,
legend style={at={(0.15,0.05)},anchor=south},
]

\addplot[semithick,mark=*,mark options={scale=0.6}, color=lightgreen]
file {plotdata/rs_residue_50.txt};
\addlegendentry{$N=50$}

\addplot[semithick,mark=x,mark options={scale=0.6}, color=bordeaux]
file {plotdata/rs_residue_100.txt};
\addlegendentry{$N=100$}

 \addplot[semithick,mark=+,mark options={scale=0.6}, color=color1]
 file {plotdata/rs_residue_150.txt};
\addlegendentry{$N=150$}

 \addplot[semithick,mark=*,mark options={scale=0.6}, color=color2] 
 file {plotdata/rs_residue_200.txt};
    \addlegendentry{$N=200$}

\addplot[semithick,mark=x,mark options={scale=0.6}, color=color3] 
file {plotdata/rs_residue_300.txt};
    \addlegendentry{$N=300$}

    \addplot[semithick,mark=+,mark options={scale=0.6}, color=darkblue] 
    file {plotdata/rs_residue_500.txt};
    \addlegendentry{$N=500$}
    
\end{axis}
\end{tikzpicture}
\caption{The maximal KKT residue per iteration in solving \eqref{eq:wahba}.}
\label{fig:4}
\end{figure}
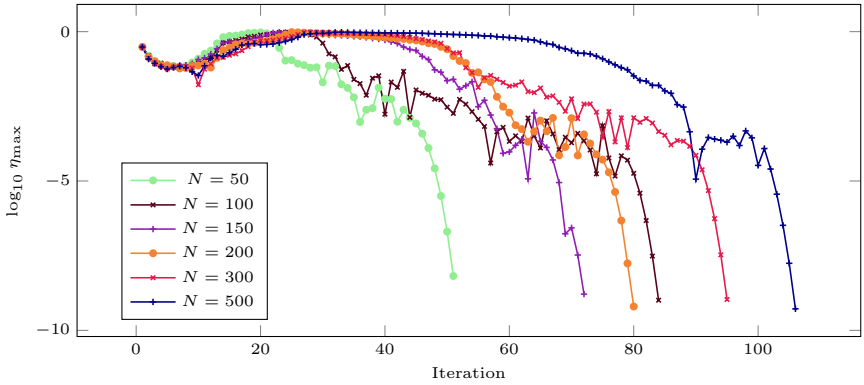  

\subsection{Nearest structured rank deficient matrices}\label{nsrd}
Let us consider the problem of finding the nearest structured rank deficient matrix:
\begin{equation}\label{eq:stlsraw}
\min_{u\in\R^{N}}\left\{\|u-\theta\|^2\middle\vert L_0 + \sum_{i=1}^Nu_iL_i\text{ is rank deficient}\right\}, 
\end{equation}
where $L_i\in\R^{s\times t} (s\le t),i=0,\dots,N$ and $\theta\in\R^{N}$ are given. Applications of Problem \eqref{eq:stlsraw} (also known as the structured total least squares problem) could be found in \cite{markovsky2008structured}. We can reformulate \eqref{eq:stlsraw} as the following polynomial optimization problem:
\begin{equation}\label{eq:stls}
\min_{z\in\R^s, u\in\R^{N}}\left\{\|u-\theta\|^2\middle\vert z^\intercal\left(L_0 + \sum_{i=1}^Nu_iL_i\right) = 0,\|z\|=1\right\}. \tag{NSRD}
\end{equation}
Note that the unit vector $z$ in \eqref{eq:stls} serves as a witness of rank deficiency. \eqref{eq:stls} is non-convex and Cifuentes proposed an SDP relaxation for \eqref{eq:stls} \cite{cifuentes2021convex} which is guaranteed to be tight under a low-noise assumption \cite{cifuentes2022local}. Let $\x=[z^{\intercal},u^{\intercal}]^{\intercal}\in\R^{s+N}$ be the vector of variables involved in \eqref{eq:stls}, and let
\begin{equation}\label{eq1:stls}
   v(\x) = [z^\intercal,u_1 z^\intercal, \dots, u_N z^\intercal]^\intercal \in \R^{s(N+1)}
\end{equation}
be the sparse set of monomials in $\x$ of degree up to two. We build $M = v(\x)v(\x)^\intercal$ as the sparse moment matrix. Then the objective function of \eqref{eq:stls} can be linearly expressed in terms of the entries of $M$. There are linear relationships among the entries of $M$: (1) all off-diagonal $s\times s$ blocks $u_iu_jzz^{\intercal}$ are symmetric; (2) each of the first $t$ equality constraint in \eqref{eq:stls}, say $g=0$, gives rise to $wg=0$ for each monomial $w\in v(\x)$; (3) the unit norm of $z$ implies that the trace of the leading $s\times s$ block of $M$ is equal to $1$. Let $\cA(X)=b$ collect all independent linear constraints obtained from these linear relationships when relaxing $M$ to an unknown PSD matrix $X$. Consequently, we obtain an SDP relaxation of size 
\begin{equation*}
    n=s(N+1),\quad m=1+st(N+1)+\frac{s(s-1)N(N+1)}{4},
\end{equation*}
which fits in \eqref{sdp1} with $\cM$ being the Euclidean manifold ($l=0$).

For each $s\in\{10,15,20,25,30,40\}$, we generate three random instances of \eqref{eq:stlsraw} with $s=t$ and $N=2s-1$. For each instance, we solve the above SDP relaxation using the solvers {\tt MOSEK}, {\tt SDPLR}, {\tt SDPNAL+}, {\tt STRIDE} and {\tt ManiSDP}, respectively.
The results are presented in Table \ref{sec6:table:nsrd} from which we can make the following conclusions. (i) {\tt MOSEK} can solve small-scale instances ($s\le15$) to high accuracy, but the running time significantly grows as $s$ increases ($\sim20$s for $s=10$ while $\sim1500$s for $s=15$). When $s\ge20$, {\tt MOSEK} runs out of space due to large memory consumption. (ii) {\tt SDPLR} can only solve small-scale instances to medium accuracy, and becomes unreliable when $s\ge20$ for returning numerical errors. (iii) {\tt SDPNAL+} is much slower than {\tt STRIDE} and {\tt ManiSDP}, and can only obtain low/medium accuracy solutions. (iv) Both {\tt STRIDE} and {\tt ManiSDP} can solve the instances to high accuracy (occasionally {\tt STRIDE} returns low/medium accuracy solutions) while {\tt ManiSDP} is faster than {\tt STRIDE} by a factor of $2\sim10$.

\begin{table}[htbp]
	\caption{Results for nearest structured rank deficient matrices.}\label{sec6:table:nsrd}
	\renewcommand\arraystretch{1.2}
	\centering
        \resizebox{\linewidth}{!}{
	\begin{tabular}{c|c|c|c|c|c|c|c|c|c|c|c}
		\multirow{2}{*}{$s$}&\multirow{2}{*}{trial}&\multicolumn{2}{c|}{{\tt MOSEK 10.0}}&\multicolumn{2}{c|}{{\tt SDPLR 1.03}}&\multicolumn{2}{c|}{{\tt SDPNAL+}}&\multicolumn{2}{c|}{{\tt STRIDE}}&\multicolumn{2}{c}{{\tt ManiSDP}}\\
		\cline{3-12}	&&$\eta_{\max}$&time&$\eta_{\max}$&time&$\eta_{\max}$&time&$\eta_{\max}$&time&$\eta_{\max}$&time\\
		\hline
		\multirow{3}{*}{10}&\#1&3.0e-11&22.9&8.4e-07&6.49&7.2e-08&64.1&3.5e-12&8.97&6.8e-10&{\bf 1.28}\\
            &\#2&4.2e-11&20.1&6.4e-05&3.04&1.8e-06&32.5&3.4e-12&4.74&4.4e-10&{\bf 1.29}\\
            &\#3&4.2e-09&15.3&6.1e-06&3.87&2.6e-05&15.5&1.3e-10&6.14&4.7e-09&{\bf 0.90}\\
            \hline
             \multirow{3}{*}{15}&\#1&4.9e-11&1623&1.5e-05&236&4.1e-06&233&4.4e-11&41.4&7.1e-09&{\bf 12.7}\\
            &\#2&3.5e-09&1436&5.0e-05&369&2.9e-03&256&1.5e-10&33.0&7.2e-09&{\bf 14.5}\\
            &\#3&4.6e-10&1558&1.1e-05&32.5&1.8e-06&151&6.2e-11&35.5&6.5e-10&{\bf 5.97}\\
            \hline
            \multirow{3}{*}{20}&\#1&-&-&$**$&$**$&3.8e-03&894&3.0e-10&174&9.7e-09&{\bf 55.9}\\
            &\#2&-&-&$**$&$**$&1.6e-02&1336&3.1e-11&125&7.9e-09&{\bf 37.5}\\
            &\#3&-&-&8.6e-06&1055&4.4e-03&1474&2.2e-10&149&7.5e-09&{\bf 40.1}\\
            \hline
            \multirow{3}{*}{25}&\#1&-&-&$**$&$**$&6.1e-03&8457&3.3e-06&4398&7.4e-09&{\bf 781}\\
            &\#2&-&-&$**$&$**$&4.4e-07&3907&5.8e-10&429&7.9e-09&{\bf 50.8}\\
            &\#3&-&-&$**$&$**$&1.3e-01&5153&2.6e-10&445&6.3e-09&{\bf 75.5}\\
            \hline
            \multirow{3}{*}{30}&\#1&-&-&$*$&$*$&$*$&$*$&4.2e-10&1812&4.9e-09&{\bf 697}\\
            &\#2&-&-&$*$&$*$&$*$&$*$&4.2e-01&2484&7.2e-09&{\bf 263}\\
            &\#3&-&-&$*$&$*$&3.8e-07&9616&3.0e-11&1042&3.9e-09&{\bf 108}\\
            \hline
            \multirow{3}{*}{40}&\#1&-&-&$*$&$*$&$*$&$*$&$*$&$*$&4.4e-09&{\bf 1984}\\
            &\#2&-&-&$*$&$*$&$*$&$*$&$*$&$*$&4.0e-09&{\bf 2493}\\
            &\#3&-&-&$*$&$*$&$*$&$*$&$*$&$*$&3.3e-09&{\bf 1279}\\
	\end{tabular}}
\end{table}

In Figures \ref{fig:7} and \ref{fig:8}, the factorization size and the maximal KKT residue per iteration in solving a random instance of \eqref{eq:stls} are shown for $s=10,15,20,25,30,40$, respectively.

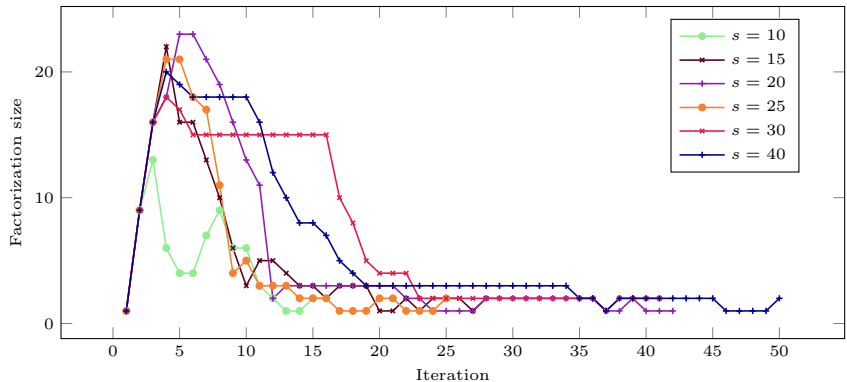
\begin{figure}[htbp]	
\centering
\begin{tikzpicture}
\footnotesize
\scalefont{0.8} 
\begin{axis}[
sharp plot, 
xmode=normal,
width=12cm, height=6cm,  
xlabel= Iteration,
ylabel = Factorization size,
xlabel near ticks,
ylabel near ticks,
legend style={at={(0.86,0.5)},anchor=south},
]

\addplot[semithick,mark=*,mark options={scale=0.6}, color=lightgreen] 
 file {plotdata/stls_fz_10.txt};
\addlegendentry{$s=10$}
 
\addplot[semithick,mark=x,mark options={scale=0.6}, color=bordeaux]
 file {plotdata/stls_fz_15.txt};
\addlegendentry{$s=15$}

 \addplot[semithick,mark=+,mark options={scale=0.6}, color=color1] 
 file {plotdata/stls_fz_20.txt};
\addlegendentry{$s=20$}

 \addplot[semithick,mark=*,mark options={scale=0.6}, color=color2]
 file {plotdata/stls_fz_25.txt};
    \addlegendentry{$s=25$}
    
\addplot[semithick,mark=x,mark options={scale=0.6}, color=color3] 
 file {plotdata/stls_fz_30.txt};
    \addlegendentry{$s=30$}

\addplot[semithick,mark=+,mark options={scale=0.6}, color=darkblue] 
 file {plotdata/stls_fz_40.txt};
    \addlegendentry{$s=40$}
    
\end{axis}
\end{tikzpicture}
\caption{The factorization size per iteration in solving \eqref{eq:stls}.}
\label{fig:7}
\end{figure}  

\begin{figure}[htbp]	
\centering
\begin{tikzpicture}
\footnotesize
\scalefont{0.8} 
\begin{axis}[
sharp plot, 
xmode=normal,
width=12cm, height=6cm,  
xlabel= Iteration,
ylabel = $\log_{10}\eta_{\max}$,
xlabel near ticks,
ylabel near ticks,
legend style={at={(0.86,0.5)},anchor=south},
]

\addplot[semithick,mark=*,mark options={scale=0.6}, color=lightgreen]
file {plotdata/stls_residue_10.txt};
\addlegendentry{$s=10$}

\addplot[semithick,mark=x,mark options={scale=0.6}, color=bordeaux]
file {plotdata/stls_residue_15.txt};
\addlegendentry{$s=15$}

 \addplot[semithick,mark=+,mark options={scale=0.6}, color=color1]
 file {plotdata/stls_residue_20.txt};
\addlegendentry{$s=20$}

 \addplot[semithick,mark=*,mark options={scale=0.6}, color=color2] 
 file {plotdata/stls_residue_25.txt};
    \addlegendentry{$s=25$}

\addplot[semithick,mark=x,mark options={scale=0.6}, color=color3] 
file {plotdata/stls_residue_30.txt};
    \addlegendentry{$s=30$}

    \addplot[semithick,mark=+,mark options={scale=0.6}, color=darkblue] 
    file {plotdata/stls_residue_40.txt};
    \addlegendentry{$s=40$}
    
\end{axis}
\end{tikzpicture}
\caption{The maximal KKT residue per iteration in solving \eqref{eq:stls}.}
\label{fig:8}
\end{figure}
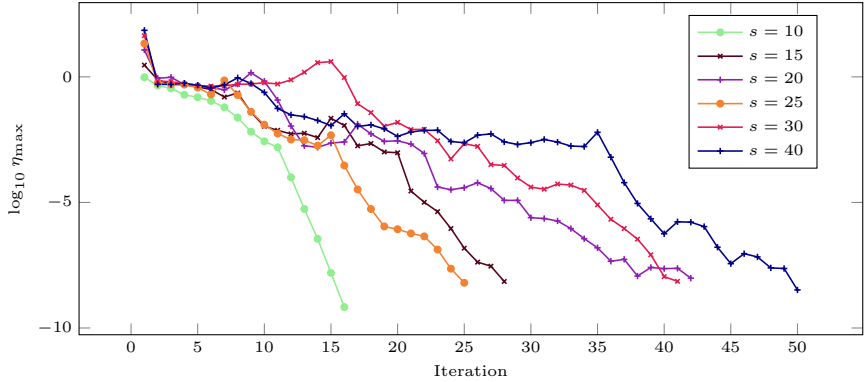  



\vspace{1em}
The numerical experiments indicate that {\tt ManiSDP} typically outperforms {\tt SDPLR} even in the case of $l = 0$ in which there is no non-trivial manifold structure to exploit. This could be explained as follows: (1) The Riemannian trust-region method enjoys superlinear (or even quadratic) convergence \cite{AbsBakGal2007-FoCM} which may guarantee fast linear convergence of the ALM whereas the ALM with L-BFGS (implemented in {\tt SDPLR}) cannot achieve linear convergence; (2) {\tt ManiSDP} implements the adaptive strategy of updating the factorization size that improves the performance a lot whereas {\tt SDPLR} utilizes a fixed factorization size ($\sim\sqrt{2m}$); (3) {\tt ManiSDP} allows to decrease the penalty parameter (note that a large penalty parameter makes the ALM subproblem more difficult to solve) whereas {\tt SDPLR} does not.

\subsection{Influence of the initial factorization size $p_0$}
In this subsection, we test the performance of {\tt ManiSDP} under different choices of the initial factorization size $p_0$ on the problems described in Sections \ref{mc}--\ref{nsrd}. For each problem of a fixed size and different choices of $p_0$, we run three random instances and then take the average running time. The results are displayed in Table \ref{p0}. It can be seen that except Problem \eqref{sec6:qs}, different choices of $p_0$ do not make big difference on the running time, while for \eqref{sec6:qs}, a larger $p_0$ leads to increment of the running time.

\begin{table}[htbp]
\caption{Running time for different choices of the initial factorization size $p_0$.}\label{p0}
\renewcommand\arraystretch{1.2}
\centering
\begin{tabular}{c|c|c|c|c|c|c|c}
\hline
\multirow{2}{*}{\eqref{sec6-eq6} ($n=4000$)}&$p_0$&1&5&10&15&20&25\\
\cline{2-8}
&time&28.4&25.6&25.8&30.9&30.0&28.1\\
\hline
\multirow{2}{*}{\eqref{sec6:bqp} ($q=50$)}&$p_0$&2&10&20&30&40&50\\
\cline{2-8}
&time&40.7&41.5&38.4&39.7&38.3&37.9\\
\hline
\multirow{2}{*}{\eqref{sec6:qs} ($q=50$)}&$p_0$&1&10&20&30&40&50\\
\cline{2-8}
&time&36.8&45.9&71.0&96.9&122&129\\
\hline
\multirow{2}{*}{\eqref{eq:wahba} ($N=150$)}&$p_0$&1&10&20&30&40&50\\
\cline{2-8}
&time&28.7&30.7&30.5&32.9&31.2&31.2\\
\hline
\multirow{2}{*}{\eqref{eq:stls} ($s=20$)}&$p_0$&1&10&20&30&40&50\\
\cline{2-8}
&time&36.8&38.0&44.9&31.8&47.0&42.9\\
\hline
\end{tabular}
\end{table}

\subsection{Comparison of {\tt ManiSDP} with and without the adaptive strategies}
In this subsection, we compare the performance of {\tt ManiSDP} with and without the adaptive strategies introduced in Sections \ref{sec5-1}--\ref{sec5-2} on the problems described in Sections \ref{mc}--\ref{nsrd}. For each problem of each size, we run three random instances and then take the average running time. The results are displayed in Table \ref{comp}, from which we see that the adaptive strategies (significantly) enhance the performance of {\tt ManiSDP} and speed up the algorithm typically by several (up to $7.5$) times.

\begin{table}[htbp]
\caption{Comparison of running time of {\tt ManiSDP} with and without the adaptive strategies. 
{\tt ManiSDP 1}: {\tt ManiSDP} with the adaptive strategies; {\tt ManiSDP 2}: {\tt ManiSDP} without the adaptive strategies.}\label{comp}
\renewcommand\arraystretch{1.2}
\centering
\begin{tabular}{c|c|c|c|c|c|c|c}
\hline
\multirow{3}{*}{\eqref{sec6-eq6}}&$n$&2000&4000&6000&8000&10000&12000\\
\cline{2-8}
&{\tt ManiSDP 1}&8.03&48.0&143&210&362&578\\
\cline{2-8}
&{\tt ManiSDP 2}&9.66&51.7&151&286&430&652\\
\hline
\multirow{3}{*}{\eqref{sec6:bqp}}&$q$&10&20&30&40&50&60\\
\cline{2-8}
&{\tt ManiSDP 1}&0.19&0.74&4.78&8.87&48.2&331\\
\cline{2-8}
&{\tt ManiSDP 2}&0.56&1.65&12.1&35.1&156&1906\\
\hline
\multirow{3}{*}{\eqref{sec6:qs}}&$q$&10&20&30&40&50&60\\
\cline{2-8}
&{\tt ManiSDP 1}&0.23&0.64&7.95&12.6&38.3&89.0\\
\cline{2-8}
&{\tt ManiSDP 2}&0.89&2.17&27.4&73.7&286&622\\
\hline
\multirow{3}{*}{\eqref{eq:wahba}}&$N$&50&100&150&200&300&500\\
\cline{2-8}
&{\tt ManiSDP 1}&3.18&20.4&34.6&66.4&193&819\\
\cline{2-8}
&{\tt ManiSDP 2}&12.0&113&163&337&765&4721\\
\hline
\multirow{3}{*}{\eqref{eq:stls}}&$s$&10&15&20&25&30&40\\
\cline{2-8}
&{\tt ManiSDP 1}&1.15&11.0&44.5&302&356&1918\\
\cline{2-8}
&{\tt ManiSDP 2}&2.67&40.1&242&2074&2238&5144\\
\hline
\end{tabular}
\end{table}

\section{Conclusions}\label{conclusions}
We have presented a manifold optimization approach to solve linear SDPs with low-rank solutions by integrating the ALM and the Burer-Monteiro factorization. 
Global convergence is guaranteed under certain conditions despite the non-convexity brought by the Burer-Monteiro factorization. A practical algorithm is provided and diverse numerical experiments demonstrate its superior performance. It has been shown that our solver {\tt ManiSDP} is capable of solving linear SDPs with millions of equality constraints to a very high precision in a reasonable time.

More research is required to achieve a comprehensive understanding of Algorithm \ref{alg1}. In particular, we believe that a global convergence result could be established under much weaker conditions. Another interesting point is the fast convergence rate of the algorithm that we empirically observed. These issues will be pursued in our future work.

We emphasize that {\tt ManiSDP} is still in an early stage of development and the strength of the approach has not been fully revealed yet. Among others, we list several directions in enhancing the approach: (1) designing a line search method to determine the step size for escaping from saddle points; (2) preconditioning for the Riemannian Hessian; (3) more efficiently escaping from saddle points; (4) handling SDPs with inequality constraints.
Moreover, as SDPs may contain multiple PSD blocks (e.g., SDP relaxations for sparse polynomial optimization problems \cite{tssos2,tssos1,tssos3}), it is also worth extending {\tt ManiSDP} to handle multi-block SDPs. We believe that all of these efforts will eventually lead to a more powerful SDP solver, which makes large-scale low-rank SDPs even more tractable and hence allows to tackle hard application problems in real world.

\section*{Acknowledgments}
The authors would like to thank Heng Yang for kindly providing the scripts for running {\tt STRIDE} and for generating random instances of the robust rotation search problem and the problem of nearest structured rank deficient matrices.

\section*{Declarations}
\subsection*{Funding}
This work is supported by National Key R\&D Program of China (No. 2022YFA1005102) and the NSFC (No. 12201618).

\subsection*{Competing interests}
The authors have no competing interests to declare that are relevant to the content of this article.

\subsection*{Data availability}
The authors confirm that all data generated or analysed during this study are included in this article.

\bibliographystyle{siamplain}
\bibliography{refer}
\end{document}